\documentclass[a4paper,12pt,centertags]{amsart}
\usepackage[utf8]{inputenc}
\usepackage{mathrsfs}
\usepackage{bbm}
\usepackage{amssymb}
\usepackage{fourier-orns}
\usepackage{eulervm}

\usepackage[T1]{fontenc}
\usepackage[a4paper,centering]{geometry}
\usepackage[colorlinks,pdfdisplaydoctitle]{hyperref}
\usepackage{tikz}
\usetikzlibrary{decorations.pathmorphing}
\usepackage{xspace}
\usepackage{mathscinet}
\newtheorem{theorem}{Theorem}[section]
\newtheorem{lemma}[theorem]{Lemma}
\newtheorem{proposition}[theorem]{Proposition}
\newtheorem{corollary}[theorem]{Corollary}
\theoremstyle{definition}
\newtheorem{definition}[theorem]{Definition}
\newtheorem{assumption}[theorem]{Assumption}
\theoremstyle{remark}
\newtheorem{remark}[theorem]{Remark}
\newtheorem{example}[theorem]{Example}
\numberwithin{equation}{section}

\newcommand{\e}{\operatorname{e}}	
\newcommand{\im}{\mathrm{i}}		
\newcommand{\N}{\mathbf{N}}		
\newcommand{\Z}{\mathbf{Z}}		
\newcommand{\R}{\mathbf{R}}		
\newcommand{\uno}{\mathbbm{1}}		
\newcommand{\E}{\mathbb{E}}		
\newcommand{\Prob}{\mathbb{P}}		
\newcommand{\field}[1][F]{\mathscr{#1}}	
\newcommand{\loc}{{\textrm{\tiny loc}}}
\newcommand{\cref}[2][]{\ensuremath{c_{\text{\ref{#2}#1}}}}
\newcommand{\supmart}{\mathcal{E}}
\newcommand{\energy}{\mathcal{G}}
\newcommand{\vf}{\mathfrak{B}}
\newcommand{\terminal}{\lower-1.8pt\hbox{\footnotesize\danger}\xspace}
\newcommand{\memo}[1]{ 
  \ensuremath{
    \framebox{\tiny\textbf{\kern-2pt\textsf{#1}}\kern-2pt}
  }
  \xspace
}
\newcommand{\arxiv}[1]{\href{http://arxiv.org/abs/#1}{arXiv:#1}}

\renewcommand{\MR}[1]{
  \href{http://www.ams.org/mathscinet-getitem?mr=#1}{\texttt{\Tiny[MR #1]}}
}

\hypersetup{%
  pdftitle={Uniqueness and blow-up for the noisy viscous dyadic model},
  pdfauthor={M. Romito},
  pdfcreator={M. Romito}
}
\begin{document}
  \title{Uniqueness and blow--up for the noisy viscous dyadic model}
  \author[M. Romito]{Marco Romito}
    \address{Dipartimento di Matematica, Universit\`a di Pisa, Largo Bruno Pontecorvo 5, I--56127 Pisa, Italia}
    \email{\href{mailto:romito@dm.unipi.it}{romito@dm.unipi.it}}
    \urladdr{\url{http://www.dm.unipi.it/pages/romito}}
  \date{October 29, 2011}
  \begin{abstract}
    We consider the dyadic model with viscosity and additive Gaussian noise
    as a simplified version of the stochastic Navier--Stokes equations, with
    the purpose of studying uniqueness and emergence of singularities. We
    prove path--wise uniqueness and absence of blow--up in the intermediate
    intensity of the non--linearity, morally corresponding to the 3D case,
    and blow--up for stronger intensity. Moreover, blow--up happens with
    probability one for regular initial data.
  \end{abstract}
\maketitle
\section{Introduction}
\subsection*{Motivations}

Uniqueness is a problem with many facet for PDEs and different problems may
require different approaches. When turning to stochastic PDEs, the problem
acquires new levels of complexity, as uniqueness for stochastic processes
can be understood in different ways (path--wise, in law, etc.). There
are several recent result on this topic (see for instance \cite{Fla11}
for a review).

A prototypical example of PDE without uniqueness are the Navier--Stokes
equations, where the issue of uniqueness is mixed with the issue of
regularity and emergence of singularities (see \cite{Fef06}). The analysis
of the stochastic version of the equations has also received a lot of
attention, and in recent years, by means of a clever way to solve
the Kolmogorov equation, Da Prato and Debussche \cite{DapDeb03} (see also
\cite{DebOda06}) have shown existence of Markov families of solutions to
the problem, proving that such Markov families admit, under suitable
assumptions on the noise, a unique invariant measure, with exponential
convergence rate \cite{Oda07}. In \cite{FlaRom06,FlaRom08} similar
results have been obtained with a completely different method, based on
the Krylov selection method \cite{Kry73}. Related results can be found
in \cite{FlaRom07,DapDeb08,Fla08,Rom08,Rom08a,Rom08b,RomXu09,Rom10a,AlbDebXu10}.

The two methods apply equally well in more general situations (as done
for instance in \cite{BloFlaRom09}). The purpose of this paper is to
analyse, with a view on uniqueness and emergence of blow--up, a much
simpler infinite dimensional stochastic equation which anyway retains most
of the characteristics of the original problem and which makes the methods
of \cite{DapDeb03} and \cite{FlaRom08} applicable. To understand which
system could provide a good candidate, the essential point is to identify
the non--linearity. If we look at the Navier--Stokes non--linearity on
the torus with periodic boundary conditions, we have in Fourier variables,
\[
  (u\cdot\nabla)u
    = \im\sum_{\mathbf{k}\in\Z^3,\mathbf{k}\neq0}
      \Bigl(
        \sum_{\mathbf{l}+\mathbf{m}=\mathbf{k},\ \mathbf{l},\mathbf{m}\neq0} (u_\mathbf{l}\cdot\mathbf{m})u_\mathbf{m}
      \Bigr)\e^{\im\mathbf{k}\cdot x},
\]
that is, the $\mathbf{k}^\text{th}$ mode interacts with almost every
other mode. The most reasonable simplification is to reduce the complexity
of the interaction to a finite number of modes, while keeping the orthogonality
property (which gives the conservation of energy), and the simplest is the
nearest neighbour interaction. This is the dyadic model.
\subsection*{The dyadic model}

The dyadic model has been introduced in \cite{FriPav04a,KatPav05} as a model
of the interaction of the energy of an inviscid fluid among different packets
of wave--modes (shells). It has been lately studied in
\cite{KisZla05,Wal06,CheFriPav07,BarFlaMor08,BarFlaMor09b} and in the
inviscid and stochastically forced case in \cite{BarFlaMor09a,BrzFlaNekZeg10,BarFlaMor10}).

The viscous version has been studied in \cite{FriPav04b,Che08,CheFri09},
and in particular \cite{Che08} proves blow--up of \emph{positive} solutions
for the problem with non--linearity of strong intensity, and later in
\cite{BarMorRom10} the authors prove well--posedness and convergence to the
inviscid limit, again for positive solutions, with non--linearity of intensity
of ``Navier--Stokes'' type.

In this paper we study the dyadic model with additive noise,
\begin{equation}\label{e:dyadic}
  dX_n = \bigl( - \nu\lambda_n^2 X_n
                + \lambda_{n-1}^\beta X_{n-1}^2
                - \lambda_n^\beta X_n X_{n+1}\bigr)\,dt
         + \sigma_n\,dW_n,
    \qquad n\geq1,
\end{equation}
where $\lambda = 2$, $\lambda_n = \lambda^n$ and $X_0\equiv0$.
The dispersion coefficients satisfy suitable assumptions (see
Assumption~\ref{a:noise}) and the parameter $\beta$ measures the relative
intensity of the non--linearity with respect to the linear term.

The non--linear term \emph{formally} satisfies the following property,
\[
  \sum_{n=1}^\infty X_n\bigl(\lambda_{n-1}X_{n-1}^2 - \lambda_n^\beta X_nX_{n+1}\bigr)
    = 0,
\]
providing an a--priori bound of $X = (X_n)_{n\geq1}$ in $\ell^2(\R)$
which is independent of $\beta$. On the other hand the linear term and
the non--linear term are roughly of the same size if
\[
  \lambda_n^2 X_n\approx\lambda_n^\beta X_n^2
    \qquad\leadsto\qquad
  \lambda_n^{\beta-2}X_n\approx O(1),
\]
which is the essential reason why local strong solutions exist when
the initial condition decays at least as $\lambda_n^{-(\beta-2)}$
(see Theorem~\ref{t:strong}). If $\beta\leq 2$ this is always true
due to the $\ell^2$ bound explained above, and in fact \cite{Che08}
proves that the non--rndom problem has a unique global solution.
The same method also works in the noisy case and it is straightforward
to prove path--wise uniqueness in the case $\beta\leq2$.

By a scaling argument (see for instance \cite{Che08}), one can
``morally'' identify the dyadic model with the Navier--Stokes equations
when $\beta\approx\tfrac52$. In \cite{BarMorRom10} well--posedness
is proved in a range which includes the value $\tfrac52$, but
only for \emph{positive} solutions (positivity is preserved by the
unforced dynamics). It is clear that, as is, positivity is broken by
the random perturbation.
\subsection*{Main results}

This paper contains a thorough analysis of the case $\beta>2$, which
can be roughly summarised in the table below.
\begin{table*}[h]
  \centering
  \begin{tabular}{|l|c|c|c|}
    \hline
               & $\beta\leq 2$ & $2<\beta\leq 3$ & $\beta>3$ \\\hline
    blow--up   & NO              & NO${}^\star$ & YES            \\\hline
    uniqueness & YES             & YES           & ?              \\\hline
  \end{tabular}
\end{table*}

We prove path--wise uniqueness (Theorem~\ref{t:unique}) in the range
$\beta\in(2,3]$ by decomposing the solution in a quasi--positive
component and a residual term. Quasi--positivity (see Section~\ref{s:negative})
is preserved by the system as long as the random perturbation is
\emph{not too strong} (the meaning of this will be explained below).
Under the same conditions the residual term is small and provides
a lower bound for the solution.

The quasi--positivity, together with the invariant area argument
of \cite{BarMorRom10}, implies smoothness of the solution
(Theorem~\ref{t:smooth}), where by smoothness we mean that
$(\lambda_n^\gamma X_n)_{n\geq1}$ is bounded for every $\gamma$.
This result holds for $\beta\in(2,\beta_c)$, where $\beta_c\in(2,3]$
is the value idenfied in \cite{BarMorRom10}.

When $\beta>3$ we use an idea of \cite{Che08}, which only works
for positive solutions, together with quasi--positivity, to
identify a set of initial conditions which leads to blow--up
with positive probability (Theorem~\ref{t:blowup}).
While there are already cases where emergence of blow--up is proved,
as for instance \cite{DebDeb02,DebDeb05} for the Schr\"odinger
equation, \cite{MueSow93,Mue00,DozLop09} for the nonlinear heat equation
(there is also \cite{BonGro09}, but their result is essentially
one dimensional and no ideas for infinite dimensional
systems are involved), such results essentially
prove only that blow--up occurs with positive probability.

Our main result on blow--up for the dyadic model, Theorem~\ref{t:noway},
states that blow--up occurs with full probability, as long as
the initial condition satisfies $\lambda_n^\alpha X_n(0) = O(1)$
for some $\alpha>\beta-2$. This is optimal since it corresponds
to the condition which ensures the existence of a local smooth solution,
as remarked before. Essentially we prove that the
$\lambda_n^{-(\beta-2)}$--decay is transient.

It remains open to understand uniqueness for $\beta>3$, since blow--up
rules out the use of smooth solutions, making path--wise uniqueness
an harder problem. Uniqueness in law may still be achievable.
\subsection*{Methods}

Our results are essentially based on the four following ideas, introduced
in this paper to analyse the model.
\subsubsection*{Quasi--positivity}

As already mentioned, the deterministic dynamics preserves positivity
and an external forcing in principle destroys this. We prove that
a negative lower bound can still be proved under the condition that the
random perturbation is not too big (in a sense clarified in the next
idea) and the negative part of the initial condition is small.
This cannot be true for general initial states, and in fact we require
this to be true only for the modes $X_n(0)$ of the initial condition
corresponding to $n$ large enough.
\subsubsection*{Irrelevance of the perturbation}

The deterministic results of \cite{Che08,BarMorRom10} are not directly
applicable due to the presence of the random term, responsible of
the loss of positivity. Quasi--positivity works with a small perturbation
and in general the random perturbation does not stay small.

We recover smallness in two ways, equally effective: on the one hand
we do not need to have a small effect of the perturbation for every
component, but only for large enough modes (in the duality
frequency/wavelength this would correspond to small scales).
On the other hand, due to viscosity, the effect of the randomness
can be quantified in terms of regularity ({i.~e.} decay in terms
of powers of $\lambda_n$) and, roughly speaking, if the effect is
finite with respect to some decay, then it is small with respect
to any weaker decay.
\subsubsection*{Contraction of the negative components}

In order to prove that blow--up happens with full probability,
the above ideas are not sufficient, as they only insure that
it happens with positive probability. We give a stronger form
of quasi--positivity (Lemma~\ref{l:contraction}), namely that
the negative parts of the solution become smaller in a finite
time, depending only on the size of the initial condition in
$H$ and on the size of the random perturbation.
\subsubsection*{Recurrence}

We use the contraction of negative components to identify
an event, spanning a time interval, thus an event for trajectories,
that leads to a set where blow--up occurs. An argument
of recurrence for these sets finally shows that blow--up has
full probability. We remark that recurrence is not at all
obvious, since for $\beta>3$ the energy estimate is not strong
enough to provide existence of a stationary solution with standard
methods for dissipative stochastic PDEs.
\subsection*{Structure of the paper}

Section \ref{s:preliminary} contains the basic definitions and assumptions,
as well as the different notions of solution that will be used, together
with some existence results. In Section~\ref{s:negative} we show
the control on negative components, which the basic result on which the
paper is structured. Uniqueness and well--posedness are then proved
in Section~\ref{s:unique} for the intermediate range $2<\beta\leq3$.
Section~\ref{s:bu_abstract} contains some preliminary and general
considerations on the blow--up time. These considerations are used
in Section~\ref{s:blowup} to first identify events that lead to blow--up,
then to prove that such sets are, conditional to the absence of blow--up,
recurrent and hence that blow--up occurs with full probability.
\subsection*{Acknowledgements}

The author wish to thank the \emph{Institut \'Elie Cartan} of the
\emph{Université Henri Poincaré, Nancy I}, where part of this
work has been done, for the kind hospitality and the exciting working
environment. Warmful thanks to David Barbato, Franco Flandoli
and Johnathan Mattingly for useful conversations and insights
on the subject. This work is dedicated to Martina, she has
arrived whilst the paper was being completed.
\section{Preliminary results and definitions}\label{s:preliminary}

We start by stating the assumptions on the strength of the noise we shall
consider.
\begin{assumption}\label{a:noise}
  The sequence $(\sigma_n)_{n\geq1}$ of non--negative real numbers satisfies
  the following assumption of regularity: there is $\alpha_0\in\R$ such that
  \begin{equation}\label{e:noise}
    \sup_{n\geq1}\bigl(\lambda_n^{\alpha_0}\sigma_n\bigr)
      < \infty.
  \end{equation}
\end{assumption}
We shall see that, in order to ensure existence of solutions for
problem \eqref{e:dyadic}, we need to impose a restriction
on the possible values of $\alpha_0$, depending on the value
of the parameter $\beta$.
\begin{assumption}\label{a:noise2}
  If $\beta$ is the parameter of problem \eqref{e:dyadic},
  then the number $\alpha_0$ of Assumption~\ref{a:noise}
  above satisfies
  \begin{equation}\label{e:noise2}
    \alpha_0>\max\bigl\{\tfrac12(\beta-3), \beta-3\bigr\}
  \end{equation}
\end{assumption}
\subsection{Notations}

Set $\lambda = 2$ and $\lambda_n = \lambda^n$.
For every $\alpha\in\R$ denote by $V_\alpha$ the (Hilbert) space
\[
  V_\alpha
    = \{(x_n)_{n\geq1}: \sum_{n=1}^\infty (\lambda_n^\alpha x_n)^2<\infty\},
\]
with scalar product $\langle\cdot,\cdot\rangle_\alpha$ and norm
$\|\cdot\|_\alpha$ given by
\[
  \langle x,y\rangle_\alpha
    = \sum_{n=1}^\infty \lambda_n^{2\alpha} x_n y_n,
      \qquad
  \|x\|_\alpha
    = \Bigl(\sum_{n=1}^\infty (\lambda_n^\alpha x_n)^2\Bigr).
\]
Set in particular $H = V_0$ and $V = V_1$.

Let $\Omega_\beta = C([0,\infty);V_{-\beta})$ and define on $\Omega_\beta$ the
canonical process $(\xi_t)_{t\geq0}$ defined as $\xi_t(\omega)=\omega(t)$,
$t\geq0$. Define on $\Omega_\beta$ the canonical filtration $(\field[B]_t)_{t\geq0}$
where $\field[B]_t$ is the Borel $\sigma$--field of $C([0,t];V_{-\beta})$.
\subsection{Definitions of solution}

We turn to the definition of solution. We shall consider first strong solutions,
which are unique, regular but defined on a (possibly) random interval.
Then we will consider weak solutions (in three different flavours).
\subsubsection{Strong solutions}

We first state the definition of local strong solution.
\begin{definition}[Strong solution]\label{d:strong}
  Let $\mathcal{W}$ be an Hilbert sub--space of $H$.
  Given a probability space $(\Omega, \field, \Prob)$
  and a cylindrical Wiener process $(W_t, \field_t)_{t\geq0}$ on $H$, 
  a \emph{strong solution in $\mathcal{W}$} with initial condition
  $x\in\mathcal{W}$ is a pair $(X(\cdot;x),\tau_x^\mathcal{W})$ such that
  \begin{itemize}
    \item $\tau_x^\mathcal{W}$ is a stopping time with $\Prob[\tau_x^\mathcal{W}>0]=1$,
    \item $X(\cdot;x)$ is a process defined on $[0,\tau_x^\mathcal{W})$ with
      $\widetilde\Prob[X(0,x)=x]=1$,
    \item $X(\cdot;x)$ is continuous with values in $\mathcal{W}$ for $t<\tau_x^\mathcal{W}$,
    \item $\|X(t;x)\|_{\mathcal{W}}\to\infty$ as $t\uparrow\tau_x^\mathcal{W}$,
      $\widetilde\Prob$--{a.~s.},
    \item $X(\cdot;x)$ is solution of \eqref{e:dyadic} on $[0,\tau_x^\mathcal{W})$.
  \end{itemize}
\end{definition}
The strong solution turns out to be a Markov process (and even a strong Markov
process, but we do not need this fact here) in the following sense (see
\cite{IkeWat81} for further details).
Set $\mathcal{W}' = \mathcal{W}\cup\{\terminal\}$, where the terminal
state \terminal is an isolated point. Define the set $\overline W(\mathcal{W}')$
of all paths $\omega:[0,\infty)\to\mathcal{W}'$ such that there exists
a time $\zeta(\omega)\in[0,\infty]$ with $\omega$ continuous with values
in $\mathcal{W}$ on $[0,\zeta(\omega))$ and $\omega(t)=\terminal$
for $t\geq\zeta(\omega)$. The strong solution defined above can be extended
as a process in $[0,\infty)$ with values in $\mathcal{W}'$ in a canonical way,
achieving value \terminal for $t\geq\tau_x^{\mathcal{W}}$. We say
that the strong solution is Markov when the process on the extended
state space $\mathcal{W}'$ is a Markov process.
\begin{theorem}\label{t:strong}
  Let $\beta>2$ and assume \eqref{e:noise}, \eqref{e:noise2}. Let
  $\alpha\in(\beta-2,\alpha_0+1)$, then for every $x\in V_\alpha$
  there exists a strong solution $(X(\cdot;x), \tau_x^\alpha)$ with initial
  condition $x$.
  
  Moreover, the solution is unique in the sense that if $(X(\cdot;x),\tau_x)$
  and $(X'(\cdot;x),\tau_x')$ are two solutions, then
  $\Prob[\tau_x=\tau_x']=1$ and $X(\cdot;x)=X'(\cdot;x)$ for $t<\tau_x$.
  
  Finally, the process $(X(\cdot;x))_{x\in V_\alpha}$ is Markov, in the sense
  given above.
\end{theorem}
\begin{proof}
  Existence and uniqueness are essentially based on the same ideas
  of \cite[Theorem 5.1]{Rom10a}, but with simpler estimates, we give
  a sketch of their proof because we will use some of the definitions
  we introduce later. Let $\chi\in C^\infty([0,\infty))$ be
  decreasing and such that $\chi(u)=1$ for $u\leq 1$ and $\chi(u)=0$
  for $u\geq2$, and consider the problem
  \begin{equation}\label{e:cutoff}
    dX_n^R = - \nu\lambda_n^2 X_n^R\,dt
             + \chi_R(\|X^R\|_\alpha)\bigl(\lambda_{n-1}^\beta (X_{n-1}^R)^2 - \lambda_n^\beta X_n^R X_{n+1}^R\bigr)\,dt
             + \sigma_n\,dW_n.
  \end{equation}
  The problem has a (path--wise) unique global solution for every $x\in V_\alpha$,
  which is continuous in time with values in $V_\alpha$. Given $x\in V_\alpha$,
  define $\tau_x^{\alpha,R}$ as the first time $t$ when $\|X^R(t)\|_\alpha=R$, then
  $\tau_x^\alpha = \sup_{R>0}\tau_x^{\alpha,R}$ and the strong solution $X(t;x)$
  coincides with $X^R(t;x)$ for $t\leq\tau_x^{\alpha,R}$. By uniqueness the definition
  makes sense. Markovianity follows by the Markovianity of each $X^R$.
\end{proof}
We notice that by path--wise uniqueness, if $x\in V_\alpha$, then
$\tau_x^\alpha=\tau_x^{\alpha'}$ for every $\alpha'\in(\beta-2,\alpha)$.
\subsubsection{Weak martingale solutions}

The fact that the \emph{blow--up time} $\tau_x^\alpha$ associated to a strong solution may
(or may not) be infinite is the main topic of discussion of the paper. In order
to have global solutions in the general case we state a weaker definition of solution.
\begin{definition}[weak martingale solution]\label{d:weak}
  Given $x\in H$, a probability measure $\Prob$ on
  $\Omega_\beta$ is a \emph{weak martingale solution} of problem~\eqref{e:dyadic}
  with initial condition $x$ if
  \begin{enumerate}
    \item for every $n\geq1$, the process
      \[
        M^n_t
          =   \xi_n(t) - \xi_n(0)
            + \int_0^t \bigl( \nu\lambda_n^2 \xi_n(s) - \lambda_{n-1}^\beta \xi_{n-1}^2 + \lambda_n^\beta \xi_n \xi_{n+1}\bigr)\,ds,
     \]
      defined through the canonical process $\xi$, is a continuous square
      integrable martingale with quadratic variation $\sigma_n^2$
      (hence in particular a Brownian motion if $\sigma_n\neq0$), 
    \item $\Prob[\xi(0) = x] = 1$.
  \end{enumerate}
\end{definition}
It can be easily seen that this is equivalent to the standard definition of
martingale solution (a proof is given for instance in \cite{Fla08} for the
Navier--Stokes equations, but can be easily adapted to this problem).

Clearly both definition, of strong and weak solution, extend straightforwardly
respectively to initial $\field_0$--measurable random variables and to initial
distributions on $H$.

It turns out (see Theorem~\ref{t:markov} below) that the strong solution
represents the unique local solution, thus all weak solutions with the
same initial condition coincide with it up to the \emph{blow--up} time
$\tau_x^\alpha$, if we require that weak solutions satisfy an additional
property. Before stating the property, we need a few definitions for
preparation.

Let $\Prob$ be a weak martingale solution and define for each $n$ such
that $\sigma_n\neq0$ the process $W_n = \tfrac1{\sigma_n}M_n$ (and
choose a $W_n$ independent from all the other if $\sigma_n=0$).
Under $\Prob$ the $(W_n)_{n\geq1}$ are independent standard
one--dimensional Brownian motions and the sequence defined as
\[
  Z_n(t) = \sigma_n\Bigl(W_n(t)
           - \nu\lambda_n^2\int_0^t\e^{-\nu\lambda_n^2(t-s)}W_n(s)\,ds\Bigr)
\]
solves the following system,
\begin{equation}\label{e:stokes}
  \begin{cases}
    dZ_n + \nu\lambda_n^2 Z_n\,dt = \sigma_n\,dW_n,\\
    Z_n(0) = 0,
  \end{cases}
    \qquad
  n\geq1.
\end{equation}
Moreover, if $Y = \xi - Z$, where $\xi$ is the canonical process
on $\Omega_\beta$, thus a solution (in the sense of Definition
\ref{d:weak}) of \eqref{e:dyadic} under $\Prob$,
then $\Prob$--{a.~s.},
\begin{equation}\label{e:v}
  \dot Y_n + \nu\lambda_n^2 Y_n
    = \lambda_{n-1}^\beta(Y_{n-1}+Z_{n-1})^2 - \lambda_n^\beta(Y_n+Z_n)(Y_{n+1}+Z_{n+1}).
\end{equation}
Define $\energy_t$ as
\begin{multline*}
  \energy_t(y,z)
    =   \|y(t)\|_H^2
      + 2\nu\int_0^t\|y(s)\|_V^2\,ds - {}\\
      - 2\int_0^t\Bigl(\sum_{n=1}^\infty\lambda_n^\beta\bigl(y_n+z_n)(y_{n+1}z_n - y_n z_{n+1}\bigr)\Bigr)\,ds,
\end{multline*}
it is easy to see (using the lemma below) that $\energy_t(Y,Z)$
is finite if $Y\in L^\infty_\loc(0,\infty;H)\cap L^2_\loc(0,\infty;V)$
and Assumptions~\ref{a:noise}, \ref{a:noise2} hold, and is jointly measurable
in the variables $(t,y,z)$ (see \cite{BloFlaRom09,Rom08b} for a related problem).
We first give a regularity result for $Z$, which is standard (see
\cite{DapZab92})
\begin{lemma}\label{l:zreg}
  Assume \eqref{e:noise} with $\alpha_0\in\R$.
  Given $\alpha<\alpha_0+1$, then almost surely $Z\in C([0,T];V_\alpha)$
  for every $T>0$. Moreover, for every $\epsilon\in(0,1]$, with
  $\epsilon<\alpha_0+1-\alpha$, there are $\cref[-1.$\epsilon$]{l:zreg}>0$
  and $\cref[-2.$\epsilon$]{l:zreg}>0$,  such that for every $T>0$,
  \[
    \E\Bigl[\exp\Bigl(\frac{\cref[-2.$\epsilon$]{l:zreg}}{T^\epsilon}
        \sup_{[0,T]}\|Z(t)\|_\alpha^2\Bigr)\Bigr]
      \leq \cref[-1.$\epsilon$]{l:zreg}.
  \]
\end{lemma}
\begin{definition}[energy martingale solution]\label{d:energy}
  A weak martingale solution $\Prob$
  with initial condition $x$ is an \emph{energy} martingale solution if
  \begin{enumerate}
    \item $\Prob[Y\in L^\infty_\loc(0,\infty;H)\cap L^2_\loc([0,\infty);V)] = 1$, where
      $Y = \xi - Z$ and $Z$ is the solution to~\eqref{e:stokes} associated to
      $\Prob$,
    \item there is a set $T_\Prob\subset(0,\infty)$ of null Lebesgue measure
      such that for every $s\not\in T_\Prob$ and every $t>s$,
      \[
        \Prob[\energy_t(Y,Z)\leq\energy_s(Y,Z)] = 1.
      \]
  \end{enumerate}
\end{definition}
\begin{remark}
  Assume \eqref{e:noise} with $\alpha_0>0$.
  As in \cite{FlaRom06,FlaRom07} one can consider an alternative definition
  where the almost sure energy inequality is replaced by the supermartingale
  condition given as follows:
  \begin{itemize}
    \item $\Prob[L^\infty(0,\infty;H)\cap L^2(0,\infty;V)] = 1$,
    \item for every $m\geq1$ the process
      \[
        \supmart_t^m
          =   \|\xi(t)\|_H^{2m}
            + 2m\nu\int_0^t \|\xi(s)\|_H^{2m-2}\|\xi(s)\|_1^2\,ds
            - m(2m-1)\sigma^2\int_0^t \|\xi(s)\|_H^{2m-2}\,ds
      \]
      is an \emph{almost sure supermartingale}, that is there is a set $T_\Prob\subset(0,\infty)$
      of null Lebesgue measure such that for every $s\not\in T_\Prob$ and every
      $t\geq s$,
      \[
        \E^\Prob[\supmart^m_t|\field[B]_s]
          \leq \supmart_s^m,
      \]
  \end{itemize}
  where $\sigma^2 = \sum_n \sigma_n^2$.
\end{remark}
Given $\alpha>\beta-2$ and $R>0$, define the following random times
on $\Omega_\beta$,
\begin{equation}\label{e:tauinfty}
  \tau_\infty^\alpha
    = \inf\{t\geq0: \|\omega(t)\|_\alpha=\infty\},
      \qquad
  \tau_\infty^{\alpha,R}
    = \inf\{t\geq0: \|\omega(t)\|_\alpha>R\},
\end{equation}
with the understanding that each random time is $\infty$ if
the set is empty.
\begin{theorem}\label{t:markov}
  Let $\beta>2$ and assume \eqref{e:noise}, \eqref{e:noise2}.
  \begin{itemize}
    \item For every $x\in H$ there exists at least one energy martingale
      solution $\Prob_x$.  
    \item If $\alpha\in(\beta-2,1+\alpha_0)$, $x\in V_\alpha$ and
      $\Prob_x$ is an energy martingale solution with initial
      condition $x$, then $\tau_x^\alpha = \tau_\infty^\alpha$
      under $\Prob_x$ and for every $t>0$,
      \[
       \xi_s = X(s;x),
         \quad s\leq t,\quad
         \qquad\Prob_x-a.s.\text{ on }\{\tau_x^\alpha>t\},
     \]
      where $(X(\cdot;x),\tau_x^\alpha)$ is the strong solution
      with initial condition $x$ defined on $\Omega_\beta$.
    \item There exists at least one family $(\Prob_x)_{x\in H}$ of energy
      martingale solutions satisfying the almost sure Markov property, namely
      for every $x\in H$ and every bounded measurable $\phi:H\to\R$,
      \[
       \E^{\Prob_x}\bigl[\phi(\xi_t)|\field[B]_s]
         = \E^{\Prob_{\omega(s)}}[\phi(\xi_{t-s})],
         \qquad \Prob_x-a. s.,
     \]
      for almost every $s\geq0$ (including $0$) and for all $t\geq s$.
  \end{itemize}
\end{theorem}
\begin{proof}
  The proof of the first fact can be done as in~\cite{BarBarBesFla06}.
  The proofs of the other two facts are entirely similar to those of Theorem 2.1
  of~\cite{Rom08b} and Theorem 3.6 of \cite{Rom10a}
  and we refer to these reference for further details.
\end{proof}
A natural way to prove existence of weak solution (and in fact this
is the way it is done in \cite{BarBarBesFla06}) is to use finite
dimensional approximations, namely, consider for each $N\geq1$
the solution $(X_n^{(N)})_{1\leq n\leq N}$ to the following finite
dimensional system,
\begin{equation}\label{e:galerkinx}
  \begin{cases}
    \dot X_1^{(N)} 
      = -\nu\lambda_1^2X_1^{(N)}
        - \lambda_1^\beta X_1^{(N)} X_2^{(N)}
        + \sigma_1\,dW_1,\\
    \dots,\\
    \dot X_n^{(N)}
      = - \nu\lambda_n^2X_n^{(N)}
        + \lambda_{n-1}^\beta(X_{n-1}^{(N)})^2
        - \lambda_n^\beta X_n^{(N)} X_{n+1}^{(N)}
        + \sigma_n\,dW_n,\\
    \dots,\\
    \dot X_N^{(N)}
      = - \nu\lambda_N^2X_N^{(N)}
        + \lambda_{N-1}^\beta(X_{N-1}^{(N)})^2
        + \sigma_N\,dW_N,
  \end{cases}
\end{equation}
Given $x\in H$, let $\Prob^{(N)}_x$ be the probability
distribution on $\Omega_\beta$ of the solution of the
above system with initial condition
$x^{(N)}=(x_1,x_2,\dots,x_N)$.
\begin{definition}[Galerkin martingale solution]\label{d:galerkin}
  Given $x\in H$, a \emph{Galerkin} martingale solution
  is any limit point in $\Omega_\beta$ of the sequence 
  $(\Prob^{(N)}_x)_{N\geq1}$.
\end{definition}
It is easy to verify (it is indeed the proof of existence in
Theorem~\ref{t:markov} and we refer to \cite{BarBarBesFla06}
for details in a similar problem) that \emph{Galerkin}
martingale solutions are \emph{energy} martingale solutions.
\begin{remark}
  The results stated in this section are essentially independent
  of the structure of the non--linear part and are only based
  on the fact that that the non--linearity is polynomial
  and at each mode involves only a finite number of other modes.
  
  The results of the rest of the paper are on the other hand
  strongly based on the structure of the non--linear part.
  Besides the physical motivations of \cite{FriPav04a}, from
  the point of view of studying a simplified version of Navier--Stokes,
  there are several possibilities for \emph{nearest--neighbour}
  interaction among modes. One can write \cite{KisZla05} any
  nearest--neighbour non--linearity as
  \[
    \dot X_n 
      = -\nu\lambda_n^2 X_n + (a_1 B_n^1(X) + a_2 B_n^2(X)),
  \]
  where $B_n^1$ is the one corresponding to the dyadic model and
  $B_n^2(x) = \lambda_{n+1}^\beta x_{n+1}^2 - \lambda_n^\beta x_{n-1} x_n$.
  In \cite{KisZla05} the authors notice that the inviscid problem with
  non--linearity $B_n^2$ is well--posed. Hence the dyadic model is
  the difficult, hence the most meaningful, part of any nearest--neighbour
  non--linearity.
\end{remark}
\section{Control of the negative components}\label{s:negative}

Given $\beta>2$, $\alpha\in\R$ and $c_0>0$, consider the solution
$Z$ of \eqref{e:stokes} and define the following process,
\begin{equation}\label{e:N}
  N_{\alpha,c_0}(t)
    = min\bigl\{m\geq1: |Z_n(s)|\leq c_0\nu\lambda_{n-1}^{-\alpha}
        \text{ for }s\in[0,t]\text{ and }n\geq m\bigr\},
\end{equation}
with $N_{\alpha,c_0}(t) = \infty$ if the set is empty.
\begin{lemma}[Moments of $N_{\alpha,c_0}$]\label{l:moments_N}
  Given $\beta>2$, assume \eqref{e:noise} and let $\alpha<\alpha_0+1$.
  Then for every $\gamma\in(0,\alpha_0+1-\alpha)$ and $\epsilon\in(0,1]$,
  with $\epsilon<\alpha_0+1-\alpha-\gamma$, there are two numbers
  $\cref[-1]{l:moments_N}>0$ and $\cref[-2]{l:moments_N}>0$,
  depending only on $\epsilon$, $\gamma$ and $\alpha_0$,
  such that for every $t>0$ and $n\geq1$,
  \[
    \Prob[N_{\alpha,c_0}(t) > n]
      \leq \cref[-1]{l:moments_N}
        \e^{-\cref[-2]{l:moments_N}\frac{c_0\nu}{t^\epsilon}\lambda_n^\gamma}.
  \]
  In particular,
  \[
    \E\bigl[\e^{\lambda_{N_{\alpha,c_0}(t)}^\gamma}\bigr]
      <\infty,
  \]
  and $\Prob[N_{\alpha,c_0}(t) = n] > 0$ for every $n\geq1$.
\end{lemma}
\begin{proof}
  For $n\geq 1$,
  \[
    \{N_{\alpha,c_0}(t)\leq n\}
      = \Bigl\{\sup_{k\geq n}\sup_{[0,t]}\lambda_{k-1}^\alpha|Z_k(s)|\leq c_0\nu\Bigr\},
  \]
  hence if $\gamma<\alpha_0+1-\alpha$ and $k\geq n$,
  \[
    \sup_{[0,t]}\lambda_{k-1}^\alpha|Z_k(s)|
      \leq \lambda_{n-1}^{-\gamma}\sup_{[0,t]}\|Z(s)\|_{\alpha+\gamma},
  \]
  and therefore by Chebychev's inequality and Lemma~\ref{l:zreg},
  \[
    \Prob[N_{\alpha,c_0}(t)>n]
      \leq \Prob\Bigl[\sup_{[0,t]}\|Z(s)\|_{\alpha+\gamma} > c_0\nu\lambda_{n-1}^\gamma\Bigr]
      \leq \cref[-1.$\epsilon$]{l:zreg}\e^{-\cref[-2.$\epsilon$]{l:zreg}\frac{c_0\nu}{t^\epsilon}\lambda_{n-1}^\gamma},
  \]
  for every $\epsilon\in(0,1]$ with $\epsilon<\alpha_0+1-\alpha-\gamma$.
  The double--exponential moment follows easily from this estimate.

  We finally prove that $\Prob[N_{\alpha,c_0}(t)=n]>0$. We prove it for $n=1$, all
  other cases follow similarly. By independence,
  \[
    \Prob[N_{\alpha,c_0}(t)=1]
      = \exp\Bigl(- \sum_{n=1}^\infty-\log\Prob\bigl[\sup_{[0,t]}\lambda_{k-1}^\alpha|Z_k(s)|\leq c_0\nu\bigr]\Bigr)
  \]
  and it is sufficient to show that the series above is convergent. By \eqref{e:noise},
  \[
    \Prob\bigl[\sup_{[0,t]}\lambda_{k-1}^\alpha|Z_k(s)|\leq c_0\nu\bigr]
      \geq \Prob\bigl[\sup_{[0,t]}|\zeta(\lambda_n^2s)|\leq 2^\alpha c_0\nu\lambda_n^{\alpha_0+1-\alpha}\bigr],
  \]
  where $\zeta$ is the solution of the one--dimensional SDE $d\zeta + \nu\zeta\,dt = dW$,
  with $\zeta(0) = 0$, and the conclusion follows by the fact that $\alpha<1+\alpha_0$
  and standard tail estimates on the one--dimensional Ornstein--Uhlenbeck process
  (see for instance~\cite{Dir75}).
\end{proof}
The lemma below is the crucial result of the paper. To formulate its statement,
we introduce the finite dimensional approximations of the problem. Consider
problem \eqref{e:v} and, for an integer $N\geq1$, the finite dimensional
approximations of \eqref{e:v},
\begin{equation}\label{e:galerkin}
  \begin{cases}
    \dot Y_1^{(N)} = -\nu\lambda_1^2Y_1^{(N)} - \lambda_1^\beta X_1^{(N)} X_2^{(N)},\\
    \dots,\\
    \dot Y_n^{(N)} = -\nu\lambda_n^2Y_n^{(N)} +\lambda_{n-1}^\beta(X_{n-1}^{(N)})^2 - \lambda_n^\beta X_n^{(N)} X_{n+1}^{(N)},\\
    \dots,\\
    \dot Y_N^{(N)} = -\nu\lambda_N^2Y_N^{(N)} +\lambda_{N-1}^\beta(X_{N-1}^{(N)})^2,
  \end{cases}
\end{equation}
where for $n=1,\dots,N$ we have set $X_n^{(N)} = Y_n^{(N)} + Z_n$. It is easy
to verify that the above SDE admits a unique global solution.
\begin{lemma}[Main lemma]\label{l:crucial}
  Let $\beta>2$ and assume \eqref{e:noise}, \eqref{e:noise2}.
  Let $\alpha\in[\beta-2,1+\alpha_0)$ and consider $C_0>0$,
  $a_0>0$ and $n_0\geq1$ such that
  \begin{equation}\label{e:crucial}
    c_0 \leq a_0
      \qquad\text{and}\qquad
    c_0 < \sqrt{a_0}\bigl(\lambda_{n_0}^{\frac12(\alpha+2-\beta)}-\sqrt{a_0}\bigr).
  \end{equation}
  Given $T>0$, let $N\geq1$ and assume that
  $\lambda_{n-1}^\alpha X_n^{(N)}(0)\geq - a_0\nu$ for all $n=n_0,\dots,N$.
  
  If $N > N_{\alpha,c_0}(T)$, then $Y_n^{(N)}(t)\geq -a_0\nu\lambda_{n-1}^{-\alpha}$
  for all $t\in[0,T]$ and all $n\geq n_0\vee N_{\alpha,c_0}(T)$.
\end{lemma}
\begin{proof}
  For simplicity we drop in this proof the superscript ${}^{(N)}$.
  We can first assume that $\lambda_{n-1}^\alpha Y_n(0)>-\nu a_0$ for
  $n\geq n_0,\dots,N$ (the equality will be included by continuity),
  then the same is true in a neighbourhood of $t=0$. Let $t_0>0$ be
  the first time when at least for one $n$,
  $\lambda_{n-1}^\alpha Y_n(t_0) = -\nu a_0$. Let
  $n\geq n_0\vee N_{\alpha,c_0}(T)$ be one of such indices, then
  \[
    \begin{aligned}
      \dot Y_n(t_0)
        &\geq -\nu\lambda_n^2 Y_n(t_0)
              - \lambda_n^\beta(Y_n(t_0)+Z_n(t_0))(Y_{n+1}(t_0)+Z_{n+1}(t_0))\\
        &\geq a_0\nu^2\lambda^2\lambda_{n-1}^{2-\alpha}
              + \lambda_n^\beta\bigl(a_0\nu\lambda_{n-1}^{-\alpha}-Z_n(t_0)\bigr)(Y_{n+1}(t_0)+Z_{n+1}(t_0))\\
        &\geq a_0\nu^2\lambda^2\lambda_{n-1}^{2-\alpha}
              - \lambda_n^\beta\bigl(a_0\nu\lambda_{n-1}^{-\alpha}-Z_n(t_0)\bigr)(Y_{n+1}(t_0)+Z_{n+1}(t_0))_-
    \end{aligned}
  \]
  since $\nu a_0\lambda_{n-1}^{-\alpha}-Z_n(t_0)\geq0$ for $n\geq N_{\alpha,c_0}(T)$,
  and where $x_- = \max(-x,0)$. We also know that
  $Y_{n+1}(t_0)\geq -a_0\nu\lambda_n^{-\alpha}$, hence
  $Y_{n+1}(t_0)+Z_{n+1}(t_0)\geq -\nu(a_0+c_0)\lambda_n^{-\alpha}$ and so
  $(Y_{n+1}(t_0)+Z_{n+1}(t_0))_-\leq \nu(a_0+c_0)\lambda_n^{-\alpha}$.
  We also have
  $a_0\nu\lambda_{n-1}^{-\alpha}-Z_n(t_0)\leq \nu(a_0+c_0)\lambda_{n-1}^{-\alpha}$,
  so in conclusion
  \[
    \dot Y_n(t_0)
      \geq \nu^2\lambda^2\lambda_{n-1}^{2-\alpha}
        \bigl(a_0 - \lambda_n^{\beta-2-\alpha}(a_0+c_0)^2\bigr)
      >0
  \]
  and the lemma is proved.
\end{proof}
The next theorem is based on Lemma~\ref{l:crucial} and shows
that the process can explode \emph{only} in the positive area.
\begin{theorem}
  Given $\beta>2$, assume \eqref{e:noise} and \eqref{e:noise2}.
  Let $\alpha\in(\beta-2,\alpha_0+1)$ and $x\in V_\alpha$, and
  let $(X(\cdot;x),\tau_x^\alpha)$ be the strong solution in $V_\alpha$
  with initial condition $x$. Then for every $T>0$ and $p\geq1$,
  \[
    \E\Bigl[\sup_{n\geq1}\sup_{t\in[0,T\wedge\tau_x^\alpha]}
        \bigl(\lambda_{n-1}^\alpha\bigl(X_n(t)\bigr)_-\bigr)^p\Bigr]
      <\infty.
  \]
  In particular,
  \[
    \inf_{n\geq1}\inf_{t\in [0,\tau_x^\alpha\wedge T]}\lambda_{n-1}^\alpha X_n
      > -\infty,
      \qquad \Prob\text{--a.~s.}
  \]
\end{theorem}
\begin{proof}
  Fix $x\in V_\alpha$ and $T>0$, and set $a_0=\tfrac14$ and $c_0=\tfrac16$,
  so that condition~\eqref{e:crucial} holds for any $n_0$. Next, choose $n_0\geq1$
  as the smallest integer such that $\lambda_{n-1}^\alpha x_n\geq -\frac14\nu$
  for all $n\geq n_0$. With the choice $c_0=\tfrac16$, set
  $\mathcal{Z}_{\alpha,T} = \{N_{\alpha,\frac16}(T)<\infty\}$, which by Lemma
  \ref{l:moments_N} is an event of probability one.
  
  Lemma~\ref{l:crucial} implies that on $\{\tau_x^\alpha>T\}$,
  \[
    Y_n(t)\geq -\frac14\nu\lambda_{n-1}^{-\alpha},
      \qquad\text{for }n\geq n_0\vee N_{\alpha,\frac16}(T).
  \]
  Indeed, we can set $x^{(N)}=(x_1,\dots,x_N)$ and notice that
  on the event $\{\tau_x^\alpha>T\}$, problem \eqref{e:v} has
  a unique solution, hence for every $N$ the solution of \eqref{e:galerkin}
  with initial condition $x^{(N)}$ converges to the solution of
  \eqref{e:v} with initial condition $x$ (where the convergence
  is component--wise uniform in time on $[0,T]$).

  Let $N_1 = n_0\vee N_{\alpha,\frac16}(T)$, it is clear that
  $N_1$ has the same finite moments of $N_{\alpha,\frac16}(T)$,
  moreover on $\{\tau_x^\alpha>T\}$,
  \[
    \lambda_{n-1}^\alpha X_n(t)
      \geq
        \begin{cases}
          -\lambda_{N_1-1}^\alpha \sup_{t\in[0,T]}\|X(t)\|_H,
            &\qquad n<N_1,\\
          -\frac5{12}\nu
            &\qquad n\geq N_1,
        \end{cases}
  \]
  for every $n\geq1$, and so
  \[
    \sup_{n\geq1}\sup_{t\in[0,T]}\lambda_{n-1}^\alpha\bigl(X_n(t)\bigr)_-
      \leq \frac5{12}\nu + \lambda_{N_1-1}^\alpha\sup_{t\in[0,T]}\|X(t)\|_H
  \]
  From Lemma \ref{l:moments_N} and the fact that $\E[\sup_{[0,T]}\|X(t)\|_H^p]$
  is finite for every $p\geq1$, the estimate in the statement of the theorem
  readily follows.
\end{proof}
\begin{remark}
  Given an initial condition $x\in V_\alpha$, if we set
  \[
    \tau_{x,\pm}^\alpha
      = \sup\{t: sup_{n\geq1}\lambda_n^\alpha(X_n)_\pm<\infty\},
  \]
  then $\tau_x^\alpha = \min(\tau_{x,+}^\alpha,\tau_{x,-}^\alpha)$,
  and the previous theorem essentially states that
  $\tau_x^\alpha=\tau_{x,+}^\alpha$.
\end{remark}
\begin{corollary}
  Given $\beta>2$, assume \eqref{e:noise} and \eqref{e:noise2}.
  Given $\alpha\in(\beta-2,\alpha_0+1)$ and $x\in V_\alpha$,
  assume additionally either that problem \eqref{e:v}, with
  initial condition $x$, admits a unique solution, for almost every
  possible value assumed by $Z$, or that we are dealing with
  a \emph{Galerkin} solution starting in $x$.
  Then for every $T>0$ and $p\geq1$,
  \[
    \E\Bigl[\sup_{n\geq1}\sup_{t\in[0,T]}
        \bigl(\lambda_{n-1}^\alpha\bigl(X_n(t)\bigr)_-\bigr)^p\Bigr]
      <\infty.
  \]
  In particular,
  \[
    \inf_{n\geq1}\inf_{t\in [0,T]}\lambda_{n-1}^\alpha X_n
      > -\infty,
      \qquad \Prob\text{--a.~s.}
  \]
\end{corollary}
\begin{proof}
  We simply notice that in the proof of the theorem above we have used
  the piece of information $\tau_x^\alpha>T$ only to ensure that
  problem \eqref{e:v} admits a unique solution.
  
  On the other hand, if we are dealing with a Galerkin solution, then
  up to a sub--sequence we still have component--wise uniform convergence
  in time.
\end{proof}
\section{Uniqueness and regularity for \texorpdfstring{$2<\beta\leq\tfrac52$}{2<beta<=5/2}}\label{s:unique}

In this section we prove two results, the first concerning path--wise uniqueness, the 
second concerning regularity (absence of blow--up), which are essentially extensions
of the corresponding results for the case without noise. The extension is made
possible by means of the control of negative components shown in Section
\ref{s:negative}. The path--wise uniqueness result below is restricted
to \emph{Galerkin} solutions (see Definition~\ref{d:galerkin}).
\begin{theorem}[Path--wise uniqueness]\label{t:unique}
  Let $\beta\in(2,3]$ and assume that \eqref{e:noise}, \eqref{e:noise2} hold.
  Let $X(0)\in V_{\beta-2}$, then there exists a (path--wise) unique
  solution of \eqref{e:dyadic} with initial condition $X(0)$, in the
  class of \emph{Galerkin} martingale solutions.
\end{theorem}
We do not know if there is uniqueness in some larger class (energy martingale
solutions or weak martingale solutions), neither we know if a Galerkin martingale
solution may develop blow--up. By slightly restricting the range of values of
$\beta$, we do actually have an improvement over the previous result.
\begin{theorem}[Smoothness]\label{t:smooth}
  There exists $\beta_c\in (\tfrac52,3]$ such that the following
  statement holds true. Assume that \eqref{e:noise}, \eqref{e:noise2}
  hold and let $\beta\in(2,\beta_c)$ and $\alpha\in(\beta-2,1+\alpha_0)$,
  then for every $x\in V_\alpha$, $\tau_x^\alpha=\infty$. In particular,
  path--wise uniqueness holds in the class of \emph{energy} martingale
  solutions.
\end{theorem}
\subsection{The proof of Theorem~\ref{t:unique}}

The proof is based on the idea in \cite[Proposition 3.2]{BarMorRom10} which
builds up on a result in \cite{BarFlaMor09b}. Both results are proved
for positive solutions (and no noise).
\begin{proof}[Proof of Theorem~\ref{t:unique}]
  It is sufficient to show uniqueness on any finite time interval, so
  we fix an arbitrary $T>0$ and show that there is only one solution
  on $[0,T]$. We will use Lemma~\ref{l:crucial} with $c_0=\tfrac16$
  and $a_0=\tfrac14$ (so that \eqref{e:crucial} holds for any $n_0$).
  In fact, since a Galerkin martingale solution is the component--wise
  limit of finite dimensional approximations, the bounds of the lemma
  remain stable in the limit to the infinite dimensional system.
  
  Since $X(0)\in V_{\beta-2}$, we know that there is $n_0\geq1$ such
  that $\lambda_n^{\beta-2}X_n(0)\geq -\tfrac14\nu$. We set
  $N_0 = 1 + n_0\vee N_{\beta-2,\frac16}(T)$.

  Let $X^1$, $X^2$ two solutions with the same given initial condition
  $X^1(0) = X^2(0) = X(0)$. By Lemma \ref{l:crucial} we know that
  $X_n^i(t)\geq Z_n(t) - \tfrac14\nu\lambda_{n-1}^{2-\beta}$
  for $n\geq N_0$, $t\in[0,T]$ and $i=1,2$. Set now
  \[
    A_n = X_n^1 - X_n^2,
      \qquad
    B_n = \frac12\nu\lambda_{n-1}^{2-\beta} - 2 Z_n
  \]
  and
  \[
    \psi_\ell(t) = \sum_{n=1}^{N_0-1}\frac{A_n^2}{\lambda_n},
      \qquad
    \psi_{h,N}(t) = \sum_{N_0}^N\frac{A_n^2}{\lambda_n},
      \qquad
    \psi_N(t) = \psi_\ell(t) + \psi_{h,N}(t),
  \]
  and notice that $X_n^1 + X_n^1  + B_n\geq0$ if $t\in[0,T]$ and
  $n\geq N_0$. A simple computation shows that
  \begin{multline*}
    \dot A_n
      = - \nu\lambda_n^2 A_n
        + \lambda_{n-1}^\beta(X_{n-1}^1+X_{n-1}^2)A_{n-1} + {}\\
        - \frac12\lambda_n^\beta\bigl[(X_n^1+X_n^2)A_{n+1}
          + (X_{n+1}^1+X_{n+1}^2)A_n\bigr],
  \end{multline*}
  hence for $N > N_0$,
  \begin{multline*}
    \frac{d}{dt}\psi_{h,N}
        + 2\nu\sum_{n=N_0}^N \lambda_n A_n^2
      = - \sum_{n=N_0}^N \lambda_n^{\beta-1}(X_{n+1}^1 + X_{n+1}^2) A_n^2 + {}\\
        - \lambda_N^{\beta-1}(X_N^1+X_N^2)A_NA_{N+1}
        + \lambda_{N_0-1}^{\beta-1}(X_{N_0-1}^1+X_{N_0-1}^2)A_{N_0-1}A_{N_0} =\\
      = \memo{1} + \memo{2}_N + \memo{3}_{N_0}.
  \end{multline*}
  For the first term we notice that $B_{n+1}\leq \tfrac56\nu\lambda_n^{2-\beta}$,
  hence
  \[
    \memo{1}
      \leq \sum_{n=N_0}^N\lambda_n^{\beta-1}B_{n+1}A_n^2
      \leq \nu\sum_{n=N_0}^N\lambda_n A_n^2.
  \]
  For the second term,
  \[
    \sum_{N=1}^\infty \int_0^T \memo{2}_N\,dt
      \leq \sup_{[0,T]}\|X^1 + X^2\|_H\int_0^T\|A\|_{\frac12(\beta-1)}^2\,ds,
  \]
  and the quantity on the right--hand side is {a.~s.} finite since by \eqref{e:noise} 
  $Z$ is {a.~s.} in $C([0,T];V_\gamma)$ for every $\gamma<1+\alpha_0$ (hence for
  $\gamma=\tfrac12(\beta-1)$), and since by Definition~\ref{d:energy},
  $V\in L^2([0,T];V)$ (and $\tfrac12(\beta-1)\leq1$). This implies
  that {a.~s.} $\int_0^T\memo{2}_N\,dt\to0$ as $N\to\infty$.
  
  Likewise (recall that $X_0^1 = X_0^2 = 0$),
  \[
    \frac{d}{dt}\psi_\ell
      \leq - \sum_{n=1}^{N_0-1} \lambda_n^{\beta-1}(X_{n+1}^1 + X_{n+1}^2) A_n^2
           - \memo{3}_{N_0}
      \leq \lambda_{N_0-1}^\beta\bigl(\sup_{[0,T]}\|X^1 + X^2\|_H\bigr)\psi_\ell
           - \memo{3}_{N_0},
  \]
  and in conclusion
  \[
    \frac{d}{dt}\psi_N
      \leq \lambda_{N_0-1}^\beta\bigl(\sup_{[0,T]}\|X^1 + X^2\|_H\bigr)\psi_\ell
           + \memo{2}_N.
  \]
  Integrate in time (recall that $\psi_N(0) = 0$) and take the limit as $N\uparrow\infty$,
  \[
    \psi(t)
      \leq \lambda_{N_0-1}^\beta\bigl(\sup_{[0,T]}\|X^1 + X^2\|_H\bigr) \int_0^t\psi(s)\,ds
  \]
  where $\psi_N\uparrow\psi$ and $\psi(t) = \|A(t)\|_{-\frac12}^2$. Since
  the term $\lambda_{N_0-1}^\beta\bigl(\sup_{[0,T]}\|X^1 + X^2\|_H\bigr)$ is {a.~s.}
  finite, by Gronwall's lemma it follows that {a.~s.} $A(t)=0$ for all $t\in[0,T]$.
\end{proof}
\subsection{The proof of Theorem~\ref{t:smooth}}

We start by giving a minimal requirement for smoothness of solutions
to~\eqref{e:dyadic} which is analogous to the criterion developed in the
noise--less case (see~\cite{BarMorRom10}).

Given $T>0$ define the subspace $K_T$ of $\Omega_\beta$ as
\[
  K_T
    = \bigl\{\omega\in\Omega_\beta: \lim_n\bigl(\max_{t\in[0,T]}\lambda_n^{\beta-2}|\omega_n(t)|\bigr) = 0\bigr\}.
\]
Clearly $L^\infty(0,T;V_\alpha)\subset K_T$ for all $\alpha>\beta-2$.
\begin{proposition}
  Under the assumptions of Theorem~\ref{t:smooth}, let $x\in V_\alpha$.
  If $\Prob_x$ is an energy martingale solution (Definition~\ref{d:energy})
  with initial condition $x$ and $\tau_\infty^\alpha$
  is the random time defined in \eqref{e:tauinfty}, then for every $T>0$,
  \[
    \{\tau_\infty^\alpha>T\}
      = K_T,
        \qquad\text{under }\Prob_x.
  \]
\end{proposition}
\begin{proof}
  Fix $\alpha\in(\beta-2,\alpha_0+1)$, $x\in V_\alpha$ and an energy martingale
  solution $\Prob_x$ starting at $x$, and let $\tau_\infty^\alpha$ be the random
  time defined in \eqref{e:tauinfty}.
  Assume $\tau_\infty^\alpha(\omega)>T$, then there is $R_0>\|x\|_\alpha$
  such that $\tau_\infty^{\alpha,R_0}(\omega)>T$, where $\tau_\infty^{\alpha,R_0}$
  is also defined in \eqref{e:tauinfty}. In particular
  $\|\xi_t(\omega)\|_\alpha\leq R_0$ for $t\in[0,T]$. Hence
  \begin{multline*}
    \lambda_n^{\beta-2}\max_{[0,T]}|\xi_n(t,\omega)|
      \leq \lambda_n^{\beta-2-\alpha}\max_{[0,T]}|\lambda_n^\alpha\xi_n(t,\omega)|\leq\\
      \leq \lambda_n^{\beta-2-\alpha}\sup_{[0,T]}\|\xi_t(\omega)\|_\alpha
      \leq R_0\lambda_n^{\beta-2-\alpha}
  \end{multline*}
  and $\omega\in K_T$.

  Vice versa, assume that $\omega\in K_T$ and choose a decreasing sequence $(M_n)_{n\geq1}$
  such that $M_n\downarrow0$ and $\lambda_n^{\beta-2}\max_{[0,T]}|\xi_n(t)|\leq M_n$.
  Set $u_n = \lambda_n^\alpha\xi_n$ and $m_n = \max_{[0,T]}|u_n(t)|$,
  then
  \[
    \begin{aligned}
      |u_n(t)|
        &\leq   |u_n(0)|
              + \lambda_n^\alpha|Z_n(t)|
              + M_{n-1}\lambda_n^2\int_0^t\e^{-\nu\lambda_n^2(t-s)}\bigl(\lambda^{\alpha-2}|u_{n-1}| + \lambda^{2-\beta}|u_n|\bigr)\,ds\\
        &\leq   |u_n(0)|
              + \bigl(\sup_{[0,T]}\lambda_n^\alpha|Z_n(t)|\bigr)
              + \frac1\nu\lambda^{\alpha-2}M_{n-1}m_{n-1}
              + \frac1\nu\lambda^{2-\beta}M_{n-1}m_n
    \end{aligned}
  \]
  and in conclusion
  \[
    \Bigl(1 - \frac1\nu\lambda^{2-\beta}M_{n-1}\Bigr)m_n
      \leq   |u_n(0)|
           + \bigl(\sup_{[0,T]}\lambda_n^\alpha|Z_n(t)|\bigr)
           + \frac1\nu\lambda^{\alpha-2}M_{n-1}m_{n-1}.
  \]
  For $n$ large enough we have that
  \[
    M_{n-1}\leq\frac{\nu}{2\lambda^{2-\beta}+2\lambda^{\alpha-2}}
      \qquad\text{and\qquad}
    \frac{\lambda^{\alpha-2}M_{n-1}}{\nu - \lambda^{2-\beta}M_{n-1}}\leq \frac12,
  \]
  and in particular $\bigl(1 - \tfrac1\nu\lambda^{2-\beta}M_{n-1}\bigr)\geq\tfrac12$,
  hence
  \[
    m_n\leq A_n + \frac12 m_{n-1}
  \]
  where $A_n = 2\lambda_n^\alpha|x_n| + 2 \bigl(\sup_{[0,T]}\lambda_n^\alpha|Z_n(t)|\bigr)$.
  By Lemma~\ref{l:zreg} (applied to an $\alpha'>\alpha$) we have
  $\sum_n A_n^2<\infty$ with probability $1$, since
  $\sum A_n^2\leq c(\|x\|_\alpha^2 + \sup_{[0,T]}\|Z\|_{\alpha'}^2)$. Moreover
  \[
    m_n^2
      \leq A_n^2 + \frac14 m_{n-1}^2 + A_n m_{n-1}
      \leq \frac32 A_n^2 + \frac34 m_{n-1}^2,
  \]
  and by solving the recursive inequality we get $\sum_n m_n^2<\infty$, which
  in particular implies that $\tau_\infty^\alpha(\omega)>T$.
\end{proof}
In order to prove Theorem~\ref{t:smooth}, we essentially prove that, given
a smooth $x$ and $T>0$, there is a solution $\Prob_x$ starting at $x$ that
satisfies $\Prob_x[K_T]=1$, and hence is the unique solution.
\begin{proof}[Proof of Theorem~\ref{t:smooth}]
  Fix $\alpha\in(\beta-2,\alpha_0+1)$, $x\in V_\alpha$, $T>0$ and an energy
  martingale solution $\Prob_x$ starting at $x$, and let $\tau_\infty^\alpha$
  be the random time defined in \eqref{e:tauinfty}. Without
  loss of generality we can assume that $\Prob_x$ is a Galerkin solution (see
  Definition \ref{d:galerkin}). Indeed, by Theorem~\ref{t:markov} the random time
  $\tau_\infty^\alpha$ coincides {a.~s.} with the lifespan $\tau_x^\alpha$ of the
  strong solution with the same initial condition $x$, and the probability of
  the event $\{\tau_\infty^x\}$ does not depend on $\Prob_x$, but only on
  the strong solution.
  
  Since $\Prob_x$ is a Galerkin solution, there are
  $(x^{(N_k)},\Prob^{(N_k)})_{k\in\N}$ such that $x^{(N_k)}\to x$ in $H$
  and $\Prob^{(N_k)}\rightharpoonup\Prob_x$ in $\Omega_\beta$, and for each
  $k$, $\Prob^{(N_k)}$ is the solution of \eqref{e:galerkinx}, with dimension
  $N_k$ and initial condition $x^{(N_k)}$. By definition we also have that
  $x_n^{(N_k)}=x_n$ for $n\leq N_k$.

  By a standard argument (Skorokhod's theorem) there are a common probability
  space $(\widetilde\Omega,\widetilde\field,\widetilde\Prob)$ and random
  variables $X^{(N_k)}$, $X$ on $\widetilde\Omega$ with distributions
  $\Prob^{(N_k)}$, $\Prob_x$ respectively, such that, in particular,
  $X_n^{(N_k)}\to X_n$, $\widetilde\Prob$--{a.~s.}, uniformly on $[0,T]$
  for all $n\geq1$ and all $T>0$.

  Let $\epsilon>0$ be small enough such that $\alpha>\beta-2+2\epsilon$
  and $6-2\beta-3\epsilon>0$. We will use
  Lemma~\ref{l:crucial} with $a_0<\tfrac12$ (to be chosen later in the proof)
  and $c_0=\tfrac13 a_0$.
  Let $\overline n$ be the smallest integer such that
  $\lambda_{n-1}^\alpha |x_n|\leq a_0\nu$ for all $n\geq\overline n$,
  and set $N_0 = \overline n\vee N_{\beta-2+2\epsilon,c_0}(T)$.

  For each integer $n_0\geq1$ and real $M>0$ define the sets
  \[
    A_M(n_0)
      = \Bigl\{\sup_{[0,T]}\bigl(|X_{n_0-1}^{(N_k)}| + |X_{n_0}^{(N_k)}|\bigr)
                 \leq M\text { for all $k$ such that }N_k\geq n_0\Bigr\},\\
  \]
  Notice that $\widetilde\Prob\bigl[\bigcup_{M>0}A_M(n_0)\bigr] = 1$ since
  $X_n^{(N_k)}\to X_n$ uniformly, $\widetilde\Prob$--{a.~s.}, for all $n\geq1$,
  hence
  \[
    \widetilde\Prob\Bigl[\bigcup_{n_0\geq 1,\ M>0} \bigl(\{N_0 = n_0\}\cap A_M(n_0)\bigr)\Bigr]
      = 1.
  \]
  Fix $n_0\geq1$, $M>0$ and $N\geq n_0$, then everything boils down to prove
  that $K_T$ happens on $\{N_0 = n_0\}\cap A_M(n_0)$ for
  $(X_n^{(N_k)})_{n_0\leq n\leq N_k}$ uniformly in $k$. In the following we work
  path--wise for $\omega\in\{N_0 = n_0\}\cap A_M(n_0)$ and we aim to use the
  method in~\cite{BarMorRom10}, which is based on proving that the area shown
  in Figure~\ref{f:trapping} is invariant for the dynamics of the (suitably
  rescaled) solution.
  \begin{figure}[ht]
    \begin{tikzpicture}[x=30mm,y=30mm,domain=0.1:1,smooth]
      \fill[color=lightgray] (0,0) -- (0.1,0) -- (0.1,0.1) -- (0,0.1);
      \draw[->,thin] (0,0) -- (1.5,0) node [right] {\scriptsize $U_n$};
      \draw[->,thin] (0,0) -- (0,1.1) node [above] {\scriptsize $U_{n+1}$};
      \draw[thick] (0,0) -- (0, 0.6) node[left] {\scriptsize $\theta$} --%
         (0.53,1) -- (1,1) -- (1,0.66) node[right] {\scriptsize $c$};
      \draw[thick] plot (\x,{.82*(\x-0.1)*(\x-0.1)});
      \draw[thick] (0.1,0) node[below] {\scriptsize $\delta$} -- (0,0);
      \draw[densely dotted] (0,1) node[left] {\scriptsize $1$} -- (0.53,1)%
         (1,0) node[below] {\scriptsize $1$} -- (1,0.66);
      \draw (0.5,0.5) node {\large $A$};
      \draw[->,>=latex] (0,0.3)     -- (0.1,0.3) node[right] {\Tiny $\vec{n}_1$};
      \draw[->,>=latex] (0.26,0.8)  -- (0.33,0.72) node[right] {\Tiny $\vec{n}_2$};
      \draw[->,>=latex] (0.77,1)    -- (0.77,0.9) node[left] {\Tiny $\vec{n}_3$};
      \draw[->,>=latex] (1,0.84)    -- (0.9,0.84) node[below, text centered] {\Tiny $\vec{n}_4$};
      \draw[->,>=latex] (.67,0.27)  -- (0.6,0.33) node[left] {\Tiny $\vec{n}_5$};
      \draw[->,>=latex] (0.05,0)    -- (0.05,0.1) node[right] {\Tiny $\vec{n}_6$};
      \draw[->,>=latex] (-.5,.8) node [left] {\Tiny $g(x)$} to [out=-20,in=135] (.2,.8);
      \draw[<-,>=latex] (.7,.27) to [in=160,out=-45] (1.5,.27) node [right] {\Tiny $h(x)$ (or $h_\eta(x)$)};
    \end{tikzpicture}
    \caption{The invariant area.}
    \label{f:trapping}
  \end{figure}
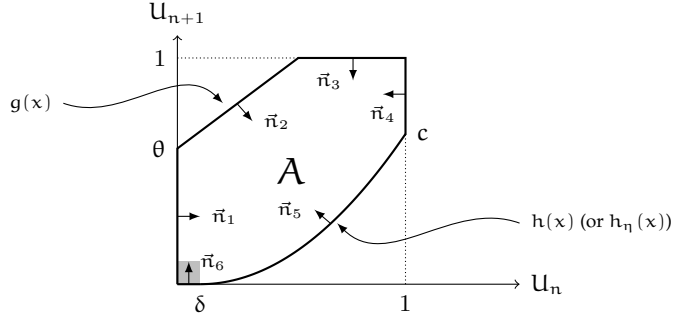
  In \cite{BarMorRom10} the area $A$ in the picture is defined for $\beta\leq\tfrac52$
  with the values $c=\lambda^{-(6-2\beta-3\epsilon)}$, $g(x)=\min(mx+\theta,1)$,
  $h(x) = c(\max(x-\delta,0)/(1-\delta))^{\lambda^2}$, $\delta=\tfrac1{10}$,
  $\theta=\tfrac35$ and $m=\tfrac34$. To deal
  with the random perturbation we shall need to slightly modify the set $A$,
  in particular we will replace the border delimited by $h(x)$ with a new
  function $h_\eta$ (see~\eqref{e:hdef}).
  
  First we define the re--scaling. Let $\epsilon_n = \nu\lambda_{n-1}^{-2\epsilon}$
  and
  \[
    \begin{cases}
      u_n = \lambda_n^\epsilon(\lambda_n^{\beta-2}Y_n^{(N_k)} + a_0\epsilon_n),\\
      v_n = \lambda_n^\epsilon(a_0\epsilon_n - \lambda_n^{\beta-2}Z_n),
    \end{cases}
  \]
  then we know by Lemma~\ref{l:crucial} that
  \begin{equation}\label{e:uvbounds}
    \begin{cases}
      u_n\geq0,\\
      \frac23a_0\lambda_n^\epsilon\epsilon_n
        \leq v_n
        \leq \frac53a_0\lambda_n^\epsilon\epsilon_n
    \end{cases}
    \qquad\text{for all }n_0\leq n\leq N_k.
  \end{equation}
  We have that
  \begin{multline*}
    \dot u_n
      = \lambda_n^{2-\epsilon}\Bigl[
          - \nu\lambda_n^\epsilon u_n
          + a_0\nu\lambda_n^{2\epsilon}\epsilon_n
          + \lambda^{\beta-4+2\epsilon}(u_{n-1} - v_{n-1})^2 + {}\\
          - \lambda^{2-\beta-\epsilon}(u_n - v_n)(u_{n+1} - v_{n+1})
        \Bigr].
  \end{multline*}
  Define $\delta_0>0$ as
  \[
    \frac1\delta_0
      = \max\Bigl\{\frac1\delta,
                   \lambda_{n_0}^{\beta-2+\epsilon}M + 2a_0\lambda_{n_0}^\epsilon\epsilon_{n_0-1},
                   \sup_{n\geq n_0}\lambda_n^\epsilon(\lambda_n^{\beta-2}x_n+2a_0\epsilon_n)\Bigr\}
  \]
  and replace $u_n$ with $U_n(t) = \delta_0^2 u_n(\delta_0 t)$. It turns
  out that $U_n(0)\leq\delta_0\leq\delta$ for all $n\geq n_0$ and 
  $\max_{[0,T]} U_{n_0-1}\leq\delta_0\leq\delta$ and
  $\max_{[0,T]} U_{n_0}\leq\delta_0\leq\delta$, since
  \[
    |U_{n_0}|
      \leq \delta_0^2\lambda_n^\epsilon\bigl(\lambda_{n_0}^{\beta-2}|X_{n_0}^{(N_k)}| + a_0\epsilon_{n_0} + \lambda_{n_0}^{\beta-2}|Z_{n_0}|\bigr)
      \leq \delta_0^2\bigl(\lambda_{n_0}^{\beta-2+\epsilon}M + 2a_0\lambda_{n_0}^\epsilon\epsilon_{n_0}\bigr)
      \leq \delta
  \]
  (and similarly for $U_{n_0-1}$). Moreover the $U_n$ satisfy
  \[
    \dot U_n
      = \lambda_n^{2-\epsilon}
         \Bigl[
             \delta_0^3 P_n^0(U,V)
           + \delta_0   P_n^1(U,V)
           + \frac1{\delta_0}\lambda^{\beta-4+2\epsilon} P_n^2(U,V)
         \Bigr],
  \]
  where
  \[
    \begin{aligned}
      P_n^0
        &=   a_0\nu\lambda_n^{2\epsilon}\epsilon_n
           + \lambda^{\beta-4+2\epsilon} V_{n-1}^2
           - \lambda^{2-\beta-\epsilon} V_n V_{n+1},\\
      P_n^1
        &= - \nu\lambda_n^\epsilon U_n
           - 2\lambda^{\beta-4+2\epsilon} V_{n-1} U_{n-1}
           + \lambda^{2-\beta-\epsilon} (V_{n+1} U_n + V_n U_{n+1}),\\
      P_n^2
        &=   U_{n-1}^2 - \lambda^{6-2\beta-3\epsilon} U_n U_{n+1},
    \end{aligned}
  \]
  and $V_n(t) = v_n(\delta_0 t)$.

  We proceed now as in \cite{BarMorRom10} and consider for each $n\geq n_0$
  the coupled systems in $(U_n,U_{n+1})$,
  \begin{equation}\label{e:field}
    \frac{d}{dt}
      \begin{pmatrix}
        U_n\\
        U_{n+1}
      \end{pmatrix}
        = \lambda_n^{2-\epsilon}\Bigl(
              \delta_0^3\vf_n^0
            + \delta_0\vf_n^1
            + \frac1{\delta_0}\lambda^{\beta-4+2\epsilon}\vf_n^2\Bigr)
  \end{equation}
  where
  \[
    \vf_n^i
      =  \begin{pmatrix}
           P_n^i\\
           \lambda^{2-\epsilon}P_{n+1}^i
         \end{pmatrix},
      \qquad i=0,1,2,
  \]
  The goal is to prove that $(U_n(t))_{n\geq n_0}$ is uniformly bounded in $n$
  and $t$. Indeed, we will show that $0\leq U_n(t)\leq 1$ for all
  $n\in\{n_0,\dots,N_k\}$, which in turns implies that
  \[
    -\lambda_n^\epsilon\epsilon_n \leq \lambda_n^{\beta-2+\epsilon}Y_n^{(N_k)} \leq \frac1{\delta_0^2}.
  \]
  for all $n\in\{n_0,\dots,N_k\}$ and all $t\in[0,T]$. Since
  $Y_n^{(N_k)}\to Y_n$ uniformly on $[0,T]$ for each $n$, the
  same bound holds for the limit $Y$ and, due to Lemma~\ref{l:zreg}, $X\in K_T$.

  It remains to show that $U_n\leq 1$, and to this end we show that each pair
  $(U_n(t), U_{n+1}(t))$, for $n\geq n_0$, remains in $A$ (as it is already
  in the interior of $A$ at time $t=0$ by the choice of $\delta_0$). This will
  be again a dynamical system proof and it is enough to prove that the vector
  field in \eqref{e:field} which gives $(\dot U_n, \dot U_{n+1})$ points
  inwards on the boundary of $A$, or equivalently that the scalar product
  of the vector fields with the normal inwards is positive.

  We make two preliminary remarks: first there is nothing to prove for the
  two pieces of boundary corresponding to the normal vectors $\vec{n}_1$ and
  $\vec{n}_6$, as this is essentially already ensured by Lemma~\ref{l:crucial},
  and second that since $A$ is convex it is also possible to check that the product
  of each of the three vector fields $\vf_n^0$, $\vf_n^1$ and $\vf_n^2$
  separately is positive.

  \emph{The vector field $\vf^1$}.
  We will use \eqref{e:uvbounds}, that $U\leq 1$ in $A$ and that $\epsilon_n$ is
  non--increasing. We will obtain lower bounds that are positive if the number
  $a_0$ is chosen small enough (with size depending only on the constants $m$,
  $\beta$ and $\epsilon$, but not on $M$, $n_0$ or $\delta_0$).
  
  On the border with normal $\vec{n}_2 = (g',1) = (m,-1)$ we have
  $U_n\in[0,\tfrac{1-\theta}{m}]$ and $U_{n+1}=mU_n+\theta$, hence
  $\lambda^2 U_{n+1} - m U_n\geq\lambda^2\theta$ and
  \begin{equation}\label{e:b12}
    \begin{aligned}
      \vf_n^1\cdot \vec{n}_2
        & =   m P_n^1 - \lambda^{2-\epsilon} P_{n+1}^1\\
        &\geq \nu\lambda_n^\epsilon(\lambda^2 U_{n+1} - m U_n) - 4a_0(m\lambda^{\beta-4+\epsilon}+\lambda^{4-\beta})\lambda_n^\epsilon\epsilon_{n-1}\\
        &\geq \lambda_n^\epsilon\bigl(\nu\lambda^2\theta - a_0 c(m,\beta,\epsilon)\epsilon_{n-1}\bigr)
    \end{aligned}
  \end{equation}
  On the border with normal $\vec{n}_3 = (0, -1)$ we have $U_{n+1} = 1$, hence
  \begin{equation}\label{e:b13}
    \vf_n^1\cdot \vec{n}_3
      =    - \lambda^{2-\epsilon} P_{n+1}^1
      \geq \lambda^{2-\epsilon} \lambda_{n+1}^\epsilon (\nu - 4a_0\lambda^{2-\beta}\epsilon_{n+1})
  \end{equation}
  Similarly, on the border with normal $\vec{n}_4 = (-1, 0)$ we have $U_n = 1$, hence
  \begin{equation}\label{e:b14}
    \vf_n^1\cdot \vec{n}_4
      =    - P_n^1
      \geq \lambda_n^\epsilon (\nu - 4a_0\lambda^{2-\beta}\epsilon_n).
  \end{equation}
  Finally we consider the border with normal $\vec{n}_5$. Prior to this we give the
  precise definition of this piece of the boundary. For $\eta\in(0,1)$ define
  $\varphi_\eta(x) = \bigl((x-\eta)/(1-\eta)\bigr)^{\lambda^2}$, $x\in[\eta,1]$,
  and notice that $\varphi_\eta$ is positive, increasing, convex and
  \[
    x\varphi_\eta' - \lambda^2\varphi_\eta
      \geq \frac{\lambda^2\eta}{1-\eta}\Bigl(\frac{\delta-\eta}{1-\eta}\Bigr)^{\lambda^2-1}
      >0.
  \]
  In \cite{BarMorRom10} it was used $h(x) = c\varphi_\delta(x)$. Here we consider
  \begin{equation}\label{e:hdef}
    h_\eta(x)
      = \frac{c}{1-\varphi_\eta(\delta)}\bigl(\varphi_\eta(x) - \varphi_\eta(\delta)\bigr),
    \qquad x\in[\delta,1],
  \end{equation}
  with $\eta<\delta$ and we notice that $h_\eta$ is positive, increasing, convex,
  $h_\eta(\delta)=0$, $h_\eta(1) = c$ and $h_\eta\to h$ in $C^1([\delta,1])$
  as $\eta\uparrow\delta$. Moreover,
  \[
    x h_\eta' - \lambda^2 h_\eta
      \geq \frac{c\lambda^2\delta}{1-\varphi_\eta(\delta)}\,\frac{(\delta-\eta)^{\lambda^2-1}}{(1-\eta)^{\lambda^2}}
      = c_{\delta-\eta}
      >0.
  \]
  With this inequality in hands, we proceed with the estimate of $\vf_n^1\cdot\vec{n}_5$.
  On the border with normal $\vec{n}_5 = (-h_\eta'(U_n), 1)$ we have $U_n\in[\delta,1]$
  and $U_{n+1} = h_\eta(U_n)$, hence
  \begin{equation}\label{e:b15}
    \begin{aligned}
      \vf_n^1\cdot\vec{n}_5
        & = \lambda^{2-\epsilon} P_{n+1}^1 - h_\eta'(U_n)P_n^1\\
        &\geq \lambda_n^\epsilon\bigl[\nu\bigl(h_\eta'(U_n)U_n - \lambda^2 h_\eta(U_n)\bigr)
                - 4a_0\lambda^{\beta-2+\epsilon}\epsilon_n - 4a_0\lambda^{2-\beta-\epsilon}h'(U_n)\epsilon_n\bigr]\\
        &\geq \lambda_n^\epsilon\bigl(\nu c_{\delta-\eta} - a_0 c(\beta,\epsilon,\delta)\epsilon_n\bigr),
    \end{aligned}
  \end{equation}
  since $h_\eta'\leq c\lambda^2 /(1-2\delta)$.

  \emph{The vector field $\vf^2$}. 
  The fact that $\vf^2_n\cdot\vec{n}_3$ and $\vf^2_n\cdot\vec{n}_4$ are positive
  follows immediately from \cite{BarMorRom10}, since in those computations the
  boundary function $h$ is not involved. As it regards $\vf^2_n\cdot\vec{n}_2$
  and $\vf^2_n\cdot \vec{n}_5$, it is easy to see that they are both
  continuous functions of $h$ and $h'$ and with the values of the parameters
  chosen in \cite{BarMorRom10} (as well as here), both scalar products have
  positive minima. Hence the same is true if we replace $h$ with $h_\eta$,
  for $\eta$ small enough, since $h_\eta\to h$ in $C^1([\delta,1])$.

  \emph{The vector field $\vf^0$}. Using~\eqref{e:uvbounds} we have
  that
  \[
    \begin{aligned}
      |P_n^0|
        &\leq a_0\nu\lambda_n^{2\epsilon}\epsilon_n + \lambda^{\beta-4+2\epsilon} V_{n-1}^2 + \lambda^{2-\beta-\epsilon} V_n V_{n+1}\\
        &\leq \lambda_n^\epsilon\bigl[a_0\nu\lambda_n^\epsilon\epsilon_n + 4a_0^2(\lambda^{\beta-4} + \lambda^{2-\beta})\lambda_n^\epsilon\epsilon_{n-1}^2\bigr]
    \end{aligned}
  \]
  and the quantity above can be made a small fraction of $\lambda_n^\epsilon$
  if $a_0$ is small enough. This ensures that together $\vf_n^1$ and $\vf_n^0$
  point inward, since the lower bounds for $\vf_n^1$ in formulae \eqref{e:b12},
  \eqref{e:b13}, \eqref{e:b14}, \eqref{e:b15} are strictly positive and
  independent of $n$, for a suitable choice of $a_0$ (up to the common
  multiplicative factor $\lambda_n^\epsilon$).

  The proof we have given (due to the choice of the numbers $m$, $\theta$, $\delta$)
  works for $\beta\leq\tfrac52$, hence we can consider $\beta_c$ slightly larger
  than $\tfrac52$. A larger value of $\beta_c$ may be considered (see 
  \cite[Remark 2.2]{BarMorRom10}).
\end{proof}
\section{The blow--up time}\label{s:bu_abstract}

In this section we analyse with more details the blow--up time introduced
in Definition~\ref{d:strong}. We give some general results which hold beyond
the dyadic model \eqref{e:dyadic} under examination and which will be
used in the next section to prove that blow--up happens with probability $1$.
Example~\ref{x:onedim} shows that the fact that blow--up happens almost
surely is a property that in general depends on the structure of the drift
and hence strongly motivates our analysis.

Thus, consider the local strong solution $(X(\cdot;x), \tau_x)_{x\in\mathcal{W}}$
of a stochastic equation (in finite or infinite dimension) on the state
space $\mathcal{W}$, where $\mathcal{W}$ is a separable Hilbert space.
By comparing with our case, we assume that
\begin{itemize}
  \item $\Prob[\tau_x>0]=1$ for all $x\in\mathcal{W}$,
  \item $X(\cdot;x)$ is continuous for $t<\tau_x$ with values in $\mathcal{W}$,
  \item $X(\cdot;x)$ is the maximal local solution, namely either $\tau_x=\infty$
    or $\|X(t;x)\|_{\mathcal{W}}\to\infty$ as $t\uparrow\infty$, $\Prob$--{a.~s.},
  \item $(X(\cdot;x), \tau_x)_{x\in\mathcal{W}}$ is Markov (in the sense
    given in Theorem~\ref{t:strong}),
  \item all martingale solutions coincide with the strong solution up to the
    blow--up time $\tau_x$.
\end{itemize}
The last statement plainly implies that the occurrence of blow--up is an intrinsic
property of the local strong solution and does not depend in an essential way
from weak solutions.

Define for $x\in\mathcal{W}$ and $t\geq0$,
\[
  \flat(t,x)
    = \Prob[\tau_x>t],
\]
then we know that $\flat(0,x)=1$ and $\flat(\cdot,x)$ is non--increasing.
Set
\[
  \flat(x)
    = \inf_{t\geq0} \flat(t,x) = \Prob[\tau_x=\infty].
\]
We prove a dichotomy for the supremum of $\flat$, namely this quantity
is either equal to $0$ or to $1$.
\begin{lemma}\label{l:solotre}
  Consider the family of processes $(X, \tau)$ on $\mathcal{W}$ as above.
  If there is $x_0\in\mathcal{W}$ such that $\Prob[\tau_{x_0}=\infty]>0$,
  then
  \[
    \sup_{x\in\mathcal{W}}\Prob[\tau_x=\infty]
      = 1
  \]
\end{lemma}
\begin{proof}
  By the Markov property,
  \[
    \flat(t+s,x)
      = \Prob[\tau_x>t+s]
      = \E[\uno_{\{\tau_x>t\}}\flat(s,X(t;x))]
  \]
  and in the limit as $s\uparrow\infty$, by monotone convergence,
  \begin{equation}\label{e:solotre}
    \flat(x)
      = \E[\uno_{\{\tau_x>t\}}\flat(X(t;x))].
  \end{equation}
  Set $c=\sup\flat(x)$, then by the above formula,
  \[
    \flat(x_0)
      =    \E[\uno_{\{\tau_{x_0}>t\}}\flat(X(t;x_0))]
      \leq c\E[\uno_{\{\tau_{x_0}>t\}}]
      =    c \flat(t,x_0),
  \]
  and, as $t\uparrow\infty$, we get $\flat(x_0)\leq c \flat(x_0)$, that is
  $c\geq1$, hence $c=1$.
\end{proof}
\begin{remark}
  If we have the additional information that there is $x_0\in\mathcal{W}$
  such that $\flat(x_0) = 1$, then something more can be said. Indeed,
  $\uno_{\{\tau_{x_0}>t\}} = 1$ almost surely, and, using again
  formula~\eqref{e:solotre},
  \[
    \E[\flat(X(t;x_0))]
      = \E[\uno_{\{\tau_{x_0}>t\}}\flat(X(t;x_0))]
      = \flat(x_0)
      = 1,
  \]
  hence $\flat(X(t;x_0)) = 1$, almost surely for every $t>0$.
  This is very close to prove that $\flat\equiv0$, and in fact
  \cite[Theorem 6.8]{FlaRom08} proves, although with a completely
  different approach, that under the assumptions of
  \begin{itemize}
    \item (suitable) strong Feller regularity,
    \item \emph{conditional irreducibility}, namely for every
      $x\in\mathcal{W}$, $t>0$ and every open set $A$ in
      $\mathcal{W}$, $\Prob[X(t;x)\in A,\tau_x>t]>0$,
  \end{itemize}
  $\flat(x_0)=1$ implies that $\flat\equiv1$ on $\mathcal{W}$.
\end{remark}
\begin{proposition}\label{p:recurrent}
  Consider the family $(X,\tau)$ of processes as above. Assume
  that, given $x\in\mathcal{W}$, there are a closed set
  $B_\infty\subset\mathcal{W}$ with non--empty interior
  and three numbers $p_0\in(0,1)$, $T_0>0$ and $T_1>0$
  such that
  \begin{itemize}
    \item $\Prob[\sigma_{B_\infty}^{x,T_1}=\infty,\tau_x=\infty]=0$,
    \item $\Prob[\tau_y\leq T_0]\geq p_0$ for every $y\in B_\infty$,
  \end{itemize}
  where the (discrete) hitting time $\sigma_{B_\infty}^{x,T_1}$ of
  $B_\infty$, starting from $x$, is defined as
  \[
    \sigma_{B_\infty}^{x,T_1}
      = \min\{k\geq0: X(kT_1;x)\in B_\infty\},
  \]
  and $\sigma_{B_\infty}^{x,T_1} = \infty$ if the set is empty.
  Then
  \[
    \Prob[\tau_x<\infty]\geq\frac{p_0}{1+p_0}.
  \]
\end{proposition}
\begin{remark}\label{r:recurrent}
  The first condition in the above proposition can be interpreted as
  a conditional recurrence: knowing that the solution does not explode,
  it will visit $B_\infty$ in a finite time with probability $1$.
\end{remark}
\begin{proof}
  The first assumption says that $\Prob[\sigma_{B_\infty}^{x,T_1}>n,\tau_x>n]\downarrow0$
  as $n\to\infty$. If $\Prob[\tau_x=\infty]=0$, then
  $\Prob[\tau_x<\infty]=1\geq\tfrac{p_0}{1+p_0}$ and there is nothing
  to prove. If on the other
  hand $\Prob[\tau_x=\infty]>0$, then $\Prob[\tau_x>n]>0$ for all $n\geq1$
  and, since $\Prob[\tau_x>n]\downarrow\Prob[\tau_x=\infty]$ as
  $n\to\infty$, we have that
  \[
    \Prob[\sigma_{B_\infty}^{x,T_1}\leq n|\tau_x>n]
      = 1 - \frac{\Prob[\sigma_{B_\infty}^{x,T_1}>n,\tau_x>n]}{\Prob[\tau_x>n]}
      \longrightarrow1,
        \qquad
      n\to\infty.
  \]
  For $n\geq1$,
  \[
    \Prob[\tau_x>n+T_0]
      \leq \Prob[\tau_x>n+T_0,\ \sigma_{B_\infty}^{x,T_1}\leq n]
           + \Prob[\sigma_{B_\infty}^{x,T_1}>n].
  \]
  The strong solution is Markov, hence
  \[
    \begin{aligned}
      \Prob[\tau_x>n+T_0,\ \sigma_{B_\infty}^{x,T_1}\leq n]
        &= \sum_{k=0}^n \Prob[\tau_x>n+T_0,\ \sigma_{B_\infty}^{x,T_1} = k]\\
        &\leq (1-p_0)\Prob[\sigma_{B_\infty}^{x,T_1}\leq n].
    \end{aligned}
  \]
  In conclusion
  \[
    \begin{aligned}
      \Prob[\tau_x>n+T_0]
        &\leq (1-p_0)\Prob[\sigma_{B_\infty}^{x,T_1}\leq n] + \Prob[\sigma_{B_\infty}^{x,T_1}>n]\\
        &=    1 - p_0\Prob[\sigma_{B_\infty}^{x,T_1}\leq n]\\
        &\leq 1 - p_0\Prob[\sigma_{B_\infty}^{x,T_1}\leq n|\tau_x>n]\Prob[\tau_x>n]
    \end{aligned}
  \]
  and in the limit as $n\to\infty$,
  \[
    \Prob[\tau_x=\infty]
      \leq 1 - p_0\Prob[\tau_x=\infty],
  \]
  that is $\Prob_x[\tau_x=\infty]\leq\tfrac1{1+p_0}$.
\end{proof}
\begin{corollary}\label{c:recurrent}
  Assume that there are $p_0\in(0,1)$, $T_0>0$ and
  $B_\infty\subset\mathcal{W}$ such that the assumptions
  of the previous proposition hold for every $x\in\mathcal{W}$
  (the time $T_1$ may depend on $x$).
  Then for every $x\in\mathcal{W}$,
  \[
    \Prob[\tau_x<\infty] = 1.
  \]
\end{corollary}
\begin{proof}
  The previous proposition yields
  \[
    \sup_{x\in\mathcal{W}}\Prob[\tau_x=\infty]
      \leq\frac{1}{1+p_0}
      <1,
  \]  
  hence the dichotomy of Lemma~\ref{l:solotre} immediately implies
  that $\Prob[\tau_x<\infty] = 1$ for every $x\in\mathcal{W}$.
\end{proof}
\begin{example}\label{x:onedim}
  Here we analyse a simple one dimensional example that shows that the fact
  that blow--up occurs with full probability does depend on the structure
  of the drift. We remark that the one dimensional case is fully understood
  (see for instance \cite{KarShr91}), our proofs below are elementary and
  mimic the proofs of the next section.

  Consider the one--dimensional SDEs
  \[
    dX = f_i(X)\,dt + dW,
      \qquad i=1,2,
  \]
  with initial condition $X(0) = x\in\R$, where 
  \[
    f_1(x) = 
      \begin{cases}
        x^2, &\quad x\geq0,\\
        x,   &\quad x<0,
      \end{cases}
        \qquad\quad
    f_2(x) = 
      \begin{cases}
        x^2, &\quad x\geq0,\\
        -x,  &\quad x<0.
      \end{cases}
  \]
  It is easy to see (using the Feller test, see for instance \cite[Proposition 5.22]{KarShr91})
  that with the drift $f_1$ the blow--up function $\flat$ defined in the
  proof of Lemma~\ref{l:solotre} verifies $0<\flat(x)<1$, while with the
  drift $f_2$ the blow--up function is $\flat(x)\equiv0$.
  
  In view of the results proved above and the analysis of the next section
  (see Theorem~\ref{t:blowup}), we observe that
  \begin{itemize}
    \item with both drifts $f_1,f_2$ and with $B_\infty = \{x\geq1\}$,
      there are $p_0\in(0,1)$ and $T_0>0$ such that
      $\Prob[\tau_x\leq T_0]\geq p_0$ for all $x\in B_\infty$ (in
      other words, the second assumption of Proposition \ref{p:recurrent}
      holds),
    \item the first assumption of Proposition \ref{p:recurrent} holds for
      $f_2$ but \emph{not} for $f_1$,
    \item in both cases $\E\bigl[\sup_{[0,T]}(X_n)_-^p\bigr]<\infty$
      for all $T>0$ and $p\geq1$.
  \end{itemize}
  We give an elementary proof of the first statement, which is modeled on
  the proof which will be used in the next section for the blow--up of
  the infinite--dimensional problem~\eqref{e:dyadic}. Given an initial
  condition $x\in [1,\infty)$, we prove that
  \[
    \Prob\Bigl[\Bigl\{\sup_{t\in [0,2]}|W_t|\leq\frac14\Bigr\}\cap\{\tau_x>2\}\Bigr]=0.
  \]
  Indeed, set $Y_t = X_t - W_t$, so that $Y_0=x$ and
  \[
    dY
      = dX - dW
      = (Y + W)^2,
  \]
  in particular $Y_t\geq1$. On the event $\{\sup_{t\in [0,2]}|W_t|\leq\tfrac14\}$,
  \[
    \dot Y
      \geq Y^2 - 2|W|Y
      \geq Y(Y-\frac12)
      \geq \frac12Y^2,
  \]
  hence by comparison $Y_t$ (and hence $X_t$) explodes before time $\tfrac{2}{x}\leq 2$.
\end{example}
\section{Blow--up for \texorpdfstring{$\beta>3$}{beta>3}}\label{s:blowup}

In this section we first prove that there are sets in the state space
which lead to a blow--up with positive probability. The underlying idea
is to use again Lemma~\ref{l:crucial} to adapt the estimates of
\cite{Che08}, which work only for positive solutions.

In the second part of the section we show that such sets are achieved
with full probability when the blow--up time is conditioned to be infinite.
The general result of the previous section immediately implies that
blow--up occurs with full probability
\subsection{Blow--up with positive probability}

Given $\alpha>\beta-2$, $p\in(0,\beta-3)$, $a_0>0$ and $M_0>0$, define the set
\begin{equation}\label{e:bu_set}
  B_\infty(\alpha,p,a_0,M_0)
    = \{x\in V_\alpha: \|x\|_p\geq M_0\text{{} and {}}
                       \inf_{n\geq1}\bigl(\lambda_{n-1}^{\beta-2}x_n\bigr)\geq -\nu a_0\}.
\end{equation}
We will show that for suitable values of $a_0$ and $M_0$ there is a positive
probability that the process solution of \eqref{e:dyadic} with initial condition
in the above set blows up in finite time.
\begin{theorem}\label{t:blowup}
  Let $\beta>3$ and assume \eqref{e:noise} and \eqref{e:noise2}.
  Given $\alpha\in(\beta-2,\alpha_0+1)$, $p\in(0,\beta-3)$, and
  $a_0\in(0,\tfrac14]$, there exist $p_0>0$, $T_0>0$ and $M_0>0$
  such that for each $x\in B_\infty(\alpha,p,a_0,M_0)$ and for every
  energy martingale weak solution $\Prob_x$ with initial condition $x$,
  \[
    \Prob_x[\tau_\infty^\alpha\leq T_0]
      \geq p_0.
  \]
\end{theorem}
\begin{proof}
  Choose $c_0>0$ such that $c_0\leq a_0$ and $c_0\leq\sqrt{a_0}(1-\sqrt{a_0})$,
  and consider the random integer $N_{\alpha,c_0}(T_0)$ defined in \eqref{e:N},
  where the value of $T_0$ will be specified at the end of the proof. Set
  \[
    p_0
      = \Prob_x[N_{\alpha,c_0}(T_0)=1],
  \]
  we remark that $p_0>0$ by Lemma~\ref{l:moments_N} and its value depends only
  on the distribution of the solution of \eqref{e:stokes}, thus is independent
  from the martingale solution $\Prob_x$. The theorem will be proved if we show
  that
  \begin{equation}\label{e:bu_claim}
    \Prob_x[\tau_\infty^\alpha>T_0,\ N_{\alpha,c_0}(T_0)=1]
      = 0,
  \end{equation}
  indeed
  \[
    \Prob_x[\tau_\infty^\alpha\leq T_0]
      = 1 - \Prob_x[\tau_\infty^\alpha\leq T_0,\ N_{\alpha,c_0}(T_0)>1]
      \geq 1 - \Prob_x[N_{\alpha,c_0}(T_0)>1]
      = p_0.
  \]
  We proceed with the proof of \eqref{e:bu_claim} and we work path--wise on the
  event
  \[
    \Omega(\alpha,T_0)
      = \{\tau_\infty^\alpha>T_0\}\cap\{N_{\alpha,c_0}(T_0)=1\}.
  \]
  Let $Z$ be the solution of \eqref{e:stokes} and $Y = X - Z$ the solution of \eqref{e:v}.
  On the event $\{\tau_\infty^\alpha>T_0\}$ \eqref{e:v} has a unique solution on $[0,T_0]$,
  and on the event $\{N_{\alpha,c_0}(T_0)=1\}$ we have that $\lambda_{n-1}^{\beta-2}|Z_n(t)|\leq c_0$
  for every $t\in[0,T_0]$ and every $n\geq1$. Set
  \[
    \eta_n = X_n - Z_n + a_0\nu\lambda_{n-1}^{2-\beta},
  \]
  then by this position $\eta = (\eta_n)_{n\geq1}$ satisfies the system
  \[
    \begin{cases}
      \dot\eta_n
        = - \nu\lambda_n^2\eta_n
          + a_0\nu^2\lambda^2 \lambda_{n-1}^{4-\beta}
          + \lambda_{n-1}^\beta X_{n-1}^2
          - \lambda_n^\beta X_n X_{n+1},\\
      \eta_n(0) = X_n(0) + a_0\nu\lambda_{n-1}^{2-\beta}.
    \end{cases}
      \qquad
    n\geq1
  \]
  Moreover, by Lemma~\ref{l:crucial} (applied with the values of $a_0$ and $c_0$
  we have fixed), it turns out that $\eta_n(t;\omega)\geq0$ for each $t\in[0,T_0]$,
  $n\geq1$ and $\omega\in\Omega(\alpha,T_0)$.

  Fix a number $b>0$, which will be specified later in formula \eqref{e:choosea},
  then
  \[
    \begin{aligned}
      \frac{d}{dt}\bigl(\eta_n^2 + b\eta_n\eta_{n+1}\bigr)
        & =    - 2\nu\lambda_n^2\eta_n^2 - b\nu(1+\lambda^2)\lambda_n^2\eta_n\eta_{n+1}\\
        &\quad + a_0\lambda^2\nu^2(2 + b\lambda^{4-\beta})\lambda_{n-1}^{4-\beta}\eta_n + a_0 b\lambda^2\nu^2\lambda_{n-1}^{4-\beta}\eta_{n+1}\\
        &\quad + 2\lambda_{n-1}^\beta X_{n-1}^2\eta_n + b\lambda_{n-1}^\beta X_{n-1}^2 \eta_{n+1} + b\lambda_n^\beta X_n^2\eta_n\\
        &\quad - 2\lambda_n^\beta X_n X_{n+1}\eta_n - b\lambda_n^\beta X_n X_{n+1}\eta_{n+1} - b\lambda_{n+1}^\beta\eta_n X_{n+1}X_{n+2}\\
        & = \memo{1} + \memo{2} + \memo{3} + \memo{4}.
    \end{aligned}
  \]
  By the choice of $a_0$ and $c_0$, we have that $(a_0 + c_0)^2\leq a_0$, hence
  \[
    \memo{2}
      \geq \lambda^2\nu^2(a_0+c_0)^2(2 + b\lambda^{4-\beta})\lambda_{n-1}^{4-\beta}\eta_n
           + b\lambda^2\nu^2(a_0 + c_0)^2\lambda_{n-1}^{4-\beta}\eta_{n+1}.
  \]
  The third term on the right--hand side of the formula above can be estimated
  from below using Young's inequality as
  \[
    \begin{aligned}
      \memo{3}
        &\geq   2\lambda_{n-1}^\beta X_{n-1}^2\eta_n
              + b\lambda_n^\beta X_n^2\eta_n\\
        &\geq \bigl[(1+\lambda^{-2p})\lambda_{n-1}^\beta\eta_{n-1}^2\eta_n - 4\nu^2(a_0+c_0)^2\frac{\lambda^{2\beta+2p-4}}{\lambda^{2p}-1}\lambda_{n-1}^{4-\beta}\eta_n\bigr]\\
        &\quad + b\bigl[\lambda_n^\beta\eta_n^3
                     - 2\lambda^{\beta-2}\nu (a_0+c_0)\lambda_n^2\eta_n^2\bigr]
    \end{aligned}
  \]
  since $X_n = \eta_n - (a_0\nu\lambda_{n-1}^{2-\beta} - Z_n)$ and
  $0\leq (a_0\nu\lambda_{n-1}^{2-\beta} - Z_n)\leq \nu(a_0+c_0)\lambda_{n-1}^{2-\beta}$.
  Likewise, the fourth term can be estimated from below as
  \[
    \begin{aligned}
      \memo{4}
        &\geq - 2\lambda_n^\beta\eta_n^2\eta_{n+1}
             - b\lambda_n^\beta\eta_n\eta_{n+1}^2
             - b\lambda_{n+1}^\beta\eta_n\eta_{n+1}\eta_{n+2}\\
        &\quad - \nu^2(a_0+c_0)^2(2\lambda^2 + b\lambda^{6-\beta})\lambda_{n-1}^{4-\beta}\eta_n
             - b\lambda^2\nu^2(a_0+c_0)^2\lambda_{n-1}^{4-\beta}\eta_{n+1},
    \end{aligned}
  \]
  where we have dropped the positive terms, using again the fact that
  $(a_0\nu\lambda_{n-1}^{2-\beta} - Z_n)\geq0$.
  The three above estimates together yield
  \[
    \frac{d}{dt}\bigl(\eta_n^2 + b\eta_n\eta_{n+1}\bigr)
        + 2\nu\lambda_n^2\eta_n^2 + b\nu(1+\lambda^2)\lambda_n^2\eta_n\eta_{n+1}
      \geq A_n + B_n + C_n,
  \]
  where
  \[
    \begin{aligned}
      A_n &=   b\lambda_n^\beta\eta_n^3
             + (1+\lambda^{-2p})\lambda_{n-1}^\beta\eta_{n-1}^2\eta_n
             - 2\lambda_n^\beta\eta_n^2\eta_{n+1}
             - b\lambda_n^\beta\eta_n\eta_{n+1}^2\\
          &\quad - b\lambda_{n+1}^\beta\eta_n\eta_{n+1}\eta_{n+2},\\
      B_n &= -2b\lambda^{\beta-2}\nu(a_0 + c_0)\lambda_n^2\eta_n^2,\\
      C_n &= - 4\nu^2(a_0+c_0)^2\frac{\lambda^{2\beta+2p-4}}{\lambda^{2p}-1}\lambda_{n-1}^{4-\beta}\eta_n.
    \end{aligned}
  \]
  By using Young's inequality in the same way as \cite{Che08},
  \begin{equation}\label{e:Anineq}
    \begin{aligned}
      A_n
        &\geq b\lambda_n^\beta\eta_n^3
          + (1+\lambda^{-2p}) \lambda_{n-1}^\beta\eta_{n-1}^2\eta_n
          - 2\lambda_n^\beta\eta_n^2\eta_{n+1}
          - \frac12 b\lambda_n^\beta \bigl(\eta_n^2\eta_{n+1} + \eta_{n+1}^3\bigr)\\
        &\quad- \frac14 b\lambda_{n+1}^\beta\bigl(2\eta_n^2\eta_{n+1} + \eta_{n+1}^2\eta_{n+2} + \eta_{n+2}^3\bigr)\\
        & = b\lambda_n^\beta\eta_n^3
          - \frac12 b\lambda_n^\beta \eta_{n+1}^3
          - \frac14 b\lambda_{n+1}^\beta\eta_{n+2}^3
          + (1+\lambda^{-2p}) \lambda_{n-1}^\beta\eta_{n-1}^2\eta_n\\
      &\quad -\bigl(2+\frac12b+\frac12\lambda^\beta b\bigr)\lambda_n^\beta\eta_n^2\eta_{n+1}
          - \frac14 b\lambda_{n+1}^\beta \eta_{n+1}^2\eta_{n+2}
    \end{aligned}
  \end{equation}
  Summing up the above estimates yields
  \begin{multline*}
    \frac{d}{dt}\Bigl[\sum_{n=1}^\infty\lambda_n^{2p}\bigl(\eta_n^2 + b\eta_n\eta_{n+1}\bigr)\Bigr]
        + 2\nu\|\eta\|_{1+p}^2 + {}\\
        + b\nu(1+\lambda^2)\sum_{n=1}^\infty\lambda_n^{2+2p}\eta_n\eta_{n+1}
      \geq \sum_{n=1}^\infty \lambda_n^{2p} (A_n + B_n + C_n).
  \end{multline*}
  By using \eqref{e:Anineq} and shifting the summation index we have
  \[
    \begin{aligned}
      \sum_{n=1}^\infty \lambda_n^{2p} A_n
        &\geq    b\Bigl(1 - \frac12\lambda^{-(\beta+2p)} - \frac14\lambda^{-(\beta+4p)}\Bigr)\sum_{n=1}^\infty \lambda_n^{\beta+2p}\eta_n^3\\
        &\quad + \Bigl[\lambda^{2p} - 1 - b\Bigl(\frac12+\frac12\lambda^\beta+\frac14\lambda^{-2p}\Bigr)\Bigr]\sum_{n=1}^\infty\lambda_n^{\beta+2p}\eta_n^2\eta_{n+1}\\
        & = k_1 \sum_{n=1}^\infty \lambda_n^{\beta+2p}\eta_n^3,
    \end{aligned}
  \]
  where we have chosen $b$ so that
  \begin{equation}\label{e:choosea}
    \lambda^{2p} - 1 - b\Bigl(\frac12+\frac12\lambda^\beta+\frac14\lambda^{-2p}\Bigr)
      = 0.
  \end{equation}
  The other two terms are simple, indeed
  \[
    \sum_{n=1}^\infty \lambda_n^{2p} B_n
      = -2b\lambda^{\beta-2}\nu(a_0+c_0)\sum_{n=1}^\infty\lambda_n^{2+2p}\eta_n^2
      = - k_2 \|\eta\|_{1+p}^2,
  \]
  and, by the Cauchy--Schwarz inequality,
  \[
    \begin{aligned}
      \sum_{n=1}^\infty \lambda_n^{2p} C_n
        &= -4\nu^2(a_0+c_0)^2\frac{\lambda^{3\beta+2p-8}}{\lambda^{2p}-1}\sum_{n=1}^\infty\lambda_n^{3-\beta+p}(\lambda_n^{1+p}\eta_n)\\
        &\geq -4\nu^2(a_0+c_0)^2\frac{\lambda^{3\beta+2p-8}}{\lambda^{2p}-1}
           \bigl(\lambda^{2(3-\beta+p)} - 1\bigr)^{-\frac12}\|\eta\|_{1+p}\\
        & = - k_3 \|\eta\|_{1+p},
    \end{aligned}
  \]
  we have used the fact that $p<\beta-3$. Since on the other hand
  \[
    \begin{aligned}
      2\nu\|\eta\|_{1+p}^2
          + b\nu(1+\lambda^2)\sum_{n=1}^\infty\lambda_n^{2+2p}\eta_n\eta_{n+1}
        &\leq   \nu\bigl(2 + b(1+\lambda^2)\lambda^{-1-p}\bigr)\|\eta\|_{1+p}^2\\
        &=      k_4\|\eta\|_{1+p}^2,
    \end{aligned}
  \]
  and
  \[
    \|\eta\|_{1+p}^2
      = \sum_{n=1}^\infty\bigl(\lambda_n^{\frac13(\beta+2p)}\eta_n\bigr)^2\lambda_n^{-\frac23(\beta-3-p)}
    \leq \Bigl(\frac1{k_5}\sum_{n=1}^\infty \lambda_n^{\beta+2p}\eta_n^3\Bigr)^{\frac23},
  \]
  all the estimates obtained so far together yield
  \[
    \dot H + k_4\psi \geq k_1k_5\psi^{\frac32} - k_2\psi - k_3\sqrt\psi,
  \]
  where
  \[
    H(t)
      = \sum_{n=1}^\infty\lambda_n^{2p}\bigl(\eta_n^2 + b\eta_n\eta_{n+1}\bigr),
        \qquad
    \psi(t)
      = \|\eta\|_{1+p}^2.
  \]
  Finally, $H\leq (1+b\lambda^{-p})\psi = k_6 \psi$, and it is easy to show
  by a simple argument (following for instance the one in \cite{Che08}) that if
  \[
    H(0)
      > M_0^2
      := \frac{k_6}{k_1 k_5}\bigl(k_4+k_2+\sqrt{(k_4+k_2)^2 + 2k_1k_3k_5}\bigr)
        \quad\text{and}\quad
    T_0
      > \frac{4k_6^{\frac32}}{k_1k_5\sqrt{H(0)}},
  \]
  the solution of the differential inequality given above for $H$ becomes
  infinite before time $T_0$.
\end{proof}
\subsection{Ineluctable occurrence of the blow--up}

The theorem in the previous section has shown that if the initial condition is
\emph{not too negative} and the noise is \emph{not too strong}, so that the
the process is not kicked away from the quasi--positive area, then the deterministic
dynamics dominates and the process explodes. In this section we show
that the sets that lead to blow--up are conditionally recurrent, in the sense
of Proposition~\ref{p:recurrent} (see Remark~\ref{r:recurrent}).
\begin{theorem}\label{t:noway}
  Let $\beta>3$ and assume \eqref{e:noise} and \eqref{e:noise2}.
  Assume moreover that the set $\{n\geq1: \sigma_n\neq0\}$ is
  non--empty. Given $\alpha\in(\beta-2,1+\alpha_0)$, for every
  $x\in V_\alpha$ and every energy martingale solution $\Prob_x$
  with initial condition $x$,
  \[
    \Prob_x[\tau_\infty^\alpha<\infty] = 1.
  \]
\end{theorem}
Before proving the theorem, we need a few preliminary results.
Our strategy is to use Corollary~\ref{c:recurrent} to prove
that blow--up happens with full probability. Hence we need to
prove that sets \eqref{e:bu_set} satisfy the assumptions
of the corollary. The idea is to find conditions on the size
of the solution in $H$ and on the structure of the perturbation,
so that the system is led to a set \eqref{e:bu_set}. Lemma
\ref{l:contraction} shows that the negative part of the solution
becomes small, while Lemma \ref{l:expansion} shows that the
size of the solution becomes large. Finally, Lemma \ref{l:recurrent}
shows that the conditions found happen with full probability,
conditional to the absence of blow--up.

We need a slight improvement of Lemma~\ref{l:crucial}. That lemma showed
that if the noise is \emph{small} from mode $N_{\alpha,c_0}(T)$ on, and if
the initial condition is \emph{not too negative} from mode $n_0$ on, then
the system is \emph{not too negative} from mode $\max(n_0,N_{\alpha,c_0}(T))$
on. In the next lemma we show that after a short time, which depends only
on the size of the initial condition \emph{in $H$}, modes becomes \emph{not
too negative} and the above result holds from mode $N_{\alpha,c_0}(T)$ on,
regardless of the value of $n_0$.

We prove the result on the event $\{N_{\alpha,c_0}(T)=1\}$, because the control
of the $H$ norm, given in the next lemma, is simpler.
\begin{lemma}\label{l:Hcontrol}
  Let $\beta>3$ and assume \eqref{e:noise} and \eqref{e:noise2}.
  There exists $\cref{l:Hcontrol}>0$ such that for $\alpha\in(\beta-2,1+\alpha_0)$,
  for every $x\in V_\alpha$, every energy martingale solution
  $\Prob_x$ starting at $x$, every $T>0$ and every $c_0>0$ with
  $4 c_0(1+\lambda^{\beta-3})\leq 1$,
  \[
    \sup_{[0,T]} \|X(t)\|_H
      \leq \|x\|_H + \cref{l:Hcontrol}\nu,
  \]
  $\Prob_x$--{a.~s.} on the event
  $\{\tau_\infty^\alpha> T\}\cap\{N_{\beta-2,c_0}(T)=1\}$.
\end{lemma}
\begin{proof}
  As already remarked, since we work on $\{\tau_\infty^\alpha> T\}$,
  problem \eqref{e:v} has a unique solution, so we can work directly
  on $Y$ (the rigorous proof proceeds through Galerkin approximations).
  We know that $\lambda_n^{\beta-2}|Z_n(t)|\leq c_0\nu$ for $t\in[0,T]$
  and $n\geq1$, hence
  \[
    \begin{aligned}
      \frac{d}{dt}\|Y\|_H^2 + 2\nu\|Y\|_1^2
        &\leq 2\sum_{n=1}^\infty \lambda_n^\beta(Y_nY_{n+1}Z_n + Y_{n+1}Z_n^2 - Y_n^2Z_{n+1} - Y_nZ_nZ_{n+1})\\
        &\leq 2c_0\nu(1+\lambda^{\beta-3})\|Y\|_1^2
              + 2c_0^2\nu^2\frac{\lambda^{\beta-2}}{\lambda^{\beta-3}-1}\|Y\|_1\\
        &\leq 4c_0\nu(1+\lambda^{\beta-3})\|Y\|_1^2
              + \frac{\lambda^{2\beta-4}}{2(\lambda^{\beta-3}-1)}c_0^3\nu^3.
    \end{aligned}
  \]
  Using the assumption on $c_0$ and the fact that $\|Y\|_1\geq\lambda\|Y\|_H$, we finally
  obtain
  \[
    \frac{d}{dt}\|Y\|_H^2 + \nu\lambda^2\|Y\|_H^2
      \leq k_0 c_0^3\nu^3,
  \]
  where the value of $k_0$ follows from the inequality above and depends only
  on $\beta$. By integrating the differential inequality, we get
  \[
    \|Y(t)\|_H^2
      \leq \|x\|_H^2 + \frac{k_0}{\lambda^2}c_0^3\nu^2,
        \quad t\in[0,T],
  \]
  and in conclusion for $t\in[0,T]$,
  \[
    \|X(t)\|_H
      \leq \|Y(t)\|_H + \|Z(t)\|_H
      \leq \|x\|_H + \bigl(\sqrt{k_0}\lambda^{-1} + (1-\lambda^{4-2\beta})^{-\frac12}\bigr)\nu,
  \]
  where we have used the fact that $c_0\leq 1$.
\end{proof}
\begin{lemma}[Contraction of the negative components]\label{l:contraction}
  Let $\beta>3$ and assume \eqref{e:noise} and \eqref{e:noise2}.
  For every $M>0$, $a_0\in(0,\frac14]$ and $c_0<a_0$, with
  $4 c_0(1+\lambda^{\beta-3})\leq 1$, there exists $T_M>0$ such
  that for every $x\in V_\alpha$, with $\alpha\in(\beta-2,1+\alpha_0)$
  and $\|x\|_H\leq M$, and every energy martingale solution $\Prob_x$,
  \[
    \inf_{n\geq 1} \bigl(\lambda_{n-1}^{\beta-2} X_n(T_M)\bigr)
      \geq -(a_0 + c_0)\nu,
  \]
  $\Prob_x$--{a.~s.} on the event $\{\tau_\infty^\alpha> T_M\}\cap\{N_{\beta-2,c_0}(T_M)=1\}$.
\end{lemma}
\begin{proof}
  Let $n_0$ be the first integer such that
  $\inf_{n\geq n_0}\lambda_n^{\beta-2}x_n\geq -a_0\nu$.
  By the choice of $c_0$ and $a_0$, we know that 
  $c_0<\sqrt{a_0}(1-\sqrt{a_0})$, hence Lemma~\ref{l:crucial}
  implies that $\lambda_n^{\beta-2}Y_n(t)\geq -a_0\nu$ holds
  for every $t\in[0,T_M]$ and every $n\geq n_0$. If $n_0=1$
  there is nothing to prove, so we consider the case $n_0>1$.

  The idea here is to show that the negative part of the mode $n_0-1$
  becomes closer to $0$ within a time $T_{n_0-1}$. After this time,
  again by Lemma~\ref{l:crucial}, the negative part of the mode $n_0-1$
  stays small, and we can apply the same contraction idea on the
  negative part of the component $n_0-2$, which then will become
  small as well within a time $T_{n_0-2}$, and so on. The sequence
  of time intervals turns out to be summable and dependent
  only on the size of the initial condition in $H$. 
    
  By these considerations it is sufficient to prove the following
  statement:
  \begin{quotation}
  given $n>1$, if we know that for a time $t_0>0$,
  \[
    \sup_{k\geq 1}\sup_{[t_0,T]}\lambda_{k-1}^{\beta-2}|Z_k|\leq c_0\nu
      \quad\text{and}\quad
    \sup_{[t_0,T]}\lambda_n^{\beta-2}(Y_{n+1})_-\leq a_0\nu
  \]
  then at time $t_0 + T_n$, where
  \[
    T_n(\|x\|_H,c_0,a_0)
      = \frac2\nu(\beta-2)\log\lambda\frac{n-1}{\lambda_n^2}
        + \frac2\nu\lambda_n^{-2}\log\Bigl(1\vee\frac{\|x\|_H+\cref{l:Hcontrol}\nu}{(a_0-c_0)\nu}\Bigr),
  \]
  we have that
  \[
    Y_n(t_0 + T_n)\geq -a_0\nu\lambda_{n-1}^{2-\beta}.
  \]
  \end{quotation}
  Before proving the claim, we notice that $\sum_n T_n<\infty$, hence we can choose
  \[
    T_M
      = \sum_{n=1}^\infty T_n(M,c_0,a_0).
  \]
  We conclude the proof of the lemma by showing the above claim. Set
  \[
    \eta_n = Y_n + c_0\nu\lambda_{n-1}^{2-\beta},
  \]
  then $X_n = \eta_n - (c_0\nu\lambda_{n-1}^{2-\beta} - Z_n)$ and
  \[
    \begin{aligned}
      \dot\eta_n
        &= -\nu\lambda_n^2\eta_n
           + c_0\nu^2\lambda^2\lambda_{n-1}^{4-\beta}
           + \lambda_{n-1}^\beta X_{n-1}^2
           - \lambda_n^\beta X_n X_{n+1}\\
        &\geq -(\nu\lambda_n^2 + \lambda_n^\beta X_{n+1})\eta_n
           + c_0\nu^2\lambda^2\lambda_{n-1}^{4-\beta}
           + \lambda_n^\beta (c_0\nu\lambda_{n-1}^{2-\beta} - Z_n)X_{n+1}\\
        &\geq -(\nu\lambda_n^2 + \lambda_n^\beta X_{n+1})\eta_n
           + c_0\nu^2\lambda^2\lambda_{n-1}^{4-\beta}
           - \lambda_n^\beta (c_0\nu\lambda_{n-1}^{2-\beta} - Z_n)(X_{n+1})_-.
    \end{aligned}
  \]
  By the assumption of the claim, $(X_{n+1})_-\leq(a_0+c_0)\nu\lambda_n^{2-\beta}$
  and $(c_0\nu\lambda_{n-1}^{2-\beta} - Z_n)\leq 2c_0\nu\lambda_{n-1}^{2-\beta}$,
  hence
  \[
    \begin{aligned}
      \dot\eta_n
        &\geq - (\nu\lambda_n^2 + \lambda_n^\beta X_{n+1})\eta_n
              + c_0\lambda^2\nu^2\bigl(1 - 2(a_0+c_0)\bigr)\lambda_{n-1}^{4-\beta}\\
        &\geq - (\nu\lambda_n^2 + \lambda_n^\beta X_{n+1})\eta_n,
    \end{aligned}
  \]
  which implies that for $t\geq t_0$,
  \begin{multline*}
    \eta_n(t)
      \geq \eta_n(t_0)\exp\Bigl(-\int_{t_0}^t(\nu\lambda_n^2 + \lambda_n^\beta X_{n+1})\,ds\Bigr)\geq\\
      \geq -(\eta_n(t_0))_-\exp\Bigl(-\int_{t_0}^t(\nu\lambda_n^2 + \lambda_n^\beta X_{n+1})\,ds\Bigr)
      \geq -(\eta_n(t_0))_-\e^{-\frac12\nu\lambda_n^2(t-t_0)},
  \end{multline*}
  since, using the fact that $a_0+c_0\leq\tfrac12$,
  \[
    \nu\lambda_n^2 + \lambda_n^\beta X_{n+1}
      \geq \nu\lambda_n^2 - \nu(a_0+c_0)\lambda_n^2
      = \nu\lambda_n^2\bigl(1 - (a_0+c_0)\bigr)
      \geq \frac12\nu\lambda_n^2.
  \]
  Finally, by Lemma~\ref{l:Hcontrol},
  \[
    (\eta_n(t_0))_-
      \leq (Y_n(t_0))_-
      \leq |Y_n(t_0)|
      \leq \|Y(t_0)\|_H
      \leq \|x\|_H + \cref{l:Hcontrol}\nu,
  \]
  and it is elementary now to check that at time $t_0+T_n$, 
  \begin{multline*}
    Y_n(t_0+T_n)
      = \eta_n(t_0+T_n) - c_0\nu\lambda_{n-1}^{2-\beta}\geq\\
      \geq -(\|x\|_H + \cref{l:Hcontrol})\e^{-\frac12\nu\lambda_n^2 T_n} - c_0\nu\lambda_{n-1}^{2-\beta}
      \geq -a_0\nu\lambda_n^{2-\beta},
  \end{multline*}
  which concludes the proof of the claim.
\end{proof}
In order to show that the hitting time for sets like \eqref{e:bu_set}, where
there is a uniform estimate on the probability of blow--up, is finite, we also
need to ensure that the size of the solution is large enough, while being
\emph{not too negative}.

At this stage the noise is crucial, although one randomly perturbed component
is enough for our purposes. The underlying ideas of the following lemma come
from control theory, although we do not need sophisticated results like
\cite{Shi06,Rom04}, because a quick and strong impulse turns out to be sufficient.
\begin{lemma}[Expansion in $H$]\label{l:expansion}
  Under the assumptions of Theorem~\ref{t:noway}, let $m$ be the smallest
  element of the set $\{n\geq1:\sigma_n\neq0\}$. Given $M_1>0$, $M_2>0$,
  and $a_0,a_0',c_0>0$ such that $c_0<a_0<a_0'<\tfrac14$ and
  $c_0+a_0<a_0'$, for every $X(0)\in V_\alpha$, with $\|X(0)\|_H\leq M_1$
  and $\inf_{n\geq 1}\lambda_{n-1}^{\beta-2} X_n(0)\geq -a_0\nu$, there
  exists $T>0$, whose value depends on $M_1$, $M_2$, $c_0$, $a_0$, $a_0'$
  and $m$, such that
  \begin{itemize}
    \item $\lambda_{n-1}^{\beta-2} X_n(t) \geq -(a_0' + c_0)\nu$ for every $n\geq1$
      and $t\in[0,T]$,
    \item $\|X(T)\|_H\geq M_2$,
  \end{itemize}
  on the event
  \[
    \{\tau_\infty^\alpha>T_M\}
      \cap\{\sup_{[0,T]}\lambda_{n-1}^{\beta-2}|Z_n(t)|\leq c_0\nu\text{ for }n\neq m\}
      \cap\{\sup_{[0,T]}\lambda_{m-1}^{\beta-2}|Z_m(t) - \psi(t)|\leq c_0\nu\},
  \]
  where $\psi:[0,T]\to\R$ is an non--decreasing continuous function such that
  $\psi(0)=0$ and $\psi(T)$ large enough depending on the above given data
  (its value is given in the proof).
\end{lemma}
\begin{proof}
  We work on the event given in the statement of the proof.
  
  \noindent\emph{Step 1: estimate in $H$.}
  We first give an estimate of $X$ in $H$. Set
  $\widetilde\psi=\sup_{[0,T]}\|Z_m\|_H\leq\psi(T) + c_0\nu$, then
  \[
    \begin{aligned}
      \frac{d}{dt}\|Y\|_H^2 + 2\nu\|Y\|_1^2
        &\leq 2\sum_{n=1}^\infty \lambda_n^\beta X_n(Z_n Y_{n+1} - Y_nZ_{n+1})\\
        &\leq 2\sum_{n=1}^\infty \lambda_n^\beta
          \bigl(|Z_n Y_n Y_{n+1}| + Z_n^2 |Y_{n+1}| + |Z_{n+1}|Y_n^2 + |Z_n Z_{n+1}Y_n|\bigr)\\
        &\quad + 2\lambda_{m-1}^\beta \bigl(|Z_m|Y_{m-1}^2 + |Z_{m-1}Z_mY_{m-1}|\bigr)\\
        &\quad + 2\lambda_m^\beta \bigl(||Z_mY_mY_{m+1} + Z_m^2 |Y_{m+1}| + |Z_mZ_{m+1}Y_m|\bigr)\\
        &\leq 4c_0\nu(1+\lambda^{\beta-3})\|Y\|_1^2 + k_0c_0^3\nu^3\\
        &\quad + \lambda_m^\beta(2\widetilde\psi^2 + 4\widetilde\psi+ 4c_0\nu\lambda^{\beta-3}\widetilde\psi)(1+\|Y\|_H^2)\\
        &\leq \nu\|Y\|_1^2 + k_0\nu^3 + 16\lambda_m^\beta(1+\nu)(1+\widetilde\psi^2)(1+\|Y\|_H^2),
    \end{aligned}
  \]
  where we have estimated the terms without $Z_m$
  as in Lemma \ref{l:Hcontrol} (the constant $k_0$
  is the same of that lemma) and the terms with $Z_m$
  bounding the components of $Y$ with $\|Y\|_H$.
  If $k_1=k_0\nu^3$, $k_2=16\lambda_m^\beta(1+\nu)$
  and $M_3(T,\psi(T))^2 = (M_1^2 + k_1/k_2)\exp\bigl(k_2T(1+\widetilde\psi^2)\bigr)$,
  it follows from Gronwall's lemma that
  $\sup_{[0,T]}\|Y(t)\|_H^2\leq M_3(T,\psi(T))^2$.
  Since on the given event we have that
  $\|Z(t)\|_H\leq \lambda^{2-\beta}(\lambda^{2-\beta}-1)^{-1} + c_0\nu + \psi(T)$
  for every $t\in[0,T]$, we finally have that
  \[
    \begin{aligned}
      \sup_{[0,T]}\|X(t)\|_H
        &\leq c_0\nu
          + \frac{\lambda^{2-\beta}}{\lambda^{2-\beta}-1}
          + \psi(T)
          + \bigl(M_1 + \sqrt{\tfrac{k_1}{k_2}}\bigr)\e^{k_2 T(1+(\psi(T)+c_0\nu)^2)}\\
        &=: M_4(T,\psi(T))
    \end{aligned}
  \]

  \noindent\emph{Step 2: Large size at time $T$.}
  Using the previous estimate we have
  \[
    \begin{aligned}
      X_m(t)
        &= \e^{-\nu\lambda_m^2t}X_m(0)
           + Z_m(t)
           + \int_0^t \e^{-\nu\lambda_m^2(t-s)}\bigl(
               \lambda_{m-1}^\beta X_{m-1}^2 - \lambda_m^\beta X_m X_{m+1}\bigr)\,ds\\
        &\geq - a_0\nu\lambda_{m-1}^{2-\beta}
              + \bigl(\psi(t) - c_0\nu\lambda_{m-1}^{2-\beta}\bigr)
              - \lambda_m^\beta t \sup_{[0,T]}\|X\|_H^2,
    \end{aligned}
  \]
  hence for $t=T$ we have
  \[
    X_m(T)
      \geq \psi(T) - \nu - \lambda_m^\beta M_4(T,\psi(T)) T.
  \]
  We can choose $\psi(T) = M_2 + 2\nu$ and notice that
  $M_4(T,M_2+2\nu) T\to 0$ as $T\downarrow0$, hence we can
  choose $T$ small enough so that
  $\lambda_m^\beta M_4(T,\psi(T)) T\leq\nu$, to obtain
  $X_m(T)\geq M_2$ and in conclusion $\|X(T)\|_H\geq M_2$.

  \noindent\emph{Step 3: Bound from below for $n=m$.}
  The choice of $\psi(T)$ and the computation in the above step yield
  \[
    \lambda_{m-1}^{\beta-2} X_m(t)
      \geq - (a_0 + c_0)\nu
           - \lambda^{2-\beta}M_4(T, M_2+2\nu)\lambda_m^{2\beta-2}  T,
  \]
  since $\psi$ is non--negative. By assumption we have that
  $a_0+c_0<a_0'$, hence, possibly fixing a smaller value of $T$ than
  the one chosen in the previous step, we can ensure that
  $X_m$ has $-a_0'\nu\lambda_{m-1}^{2-\beta}$ as lower bound on
  $[0,T]$. 

  \noindent\emph{Step 4: Bound from below for $n\neq m$.}
  If $n>m$, the proof proceeds as in Lemma~\ref{l:crucial},
  since $X_m$ appears in the system of equations for $(Y_n)_{n>m}$
  only through the \emph{positive} term $\lambda_m^\beta X_m^2$
  in the equation for the $(m+1)^\text{th}$ component (and
  since $c_0<\sqrt{a_0'}(1-\sqrt{a_0'})$ by the choice of
  $c_0$, $a_0'$).

  If $n<m$, the proof follows by finite induction. Indeed, the
  lower bound is true for $n=m$ by the previous step. Next
  we prove that if
  $\lambda_n^{2-\beta} Y_{n+1}(t)\geq -a_0'\nu$ for all $t\in[0,T]$,
  then $\lambda_{n-1}^{2-\beta} Y_n\geq -a_0'\nu$. We prove this
  similarly to the proof of Lemma~\ref{l:contraction}.
  Indeed, set $\eta_n = Y_n + a_0'\nu\lambda_{n-1}^{2-\beta}$,
  then using only that
  $\lambda_{n-1}^{\beta-2}|Z_n(t)|\leq c_0\nu$
  and $(X_{n+1}(t))_-\leq (a_0'+c_0)\nu\lambda_n^{2-\beta}$
  for $t\in[0,T]$,
  \[
    \begin{aligned}
      \dot\eta_n
        &\geq - (\nu\lambda_n^2 + \lambda_n^\beta X_{n+1})\eta_n
              + a_0'\nu^2\lambda^{\beta-2}\lambda_n^{4-\beta}
              - \lambda_n^\beta(a_0'\nu\lambda_{n-1}^{2-\beta} - Z_n)(X_{n+1})_-\\
        &\geq - (\nu\lambda_n^2 + \lambda_n^\beta X_{n+1})\eta_n
              + \nu\lambda^{\beta-2}\bigl(a_0' - (a_0'+c_0)^2\bigr)\lambda_n^{4-\beta}\\
        &\geq - (\nu\lambda_n^2 + \lambda_n^\beta X_{n+1})\eta_n,
    \end{aligned}
  \]
  since $c_0<a_0'<\tfrac14$, hence $c_0<\sqrt{a_0'}(1-\sqrt{a_0'})$.
  The fact that $\eta_n(0)\geq0$ implies that $\eta_n(t)\geq0$.
\end{proof}
Lemma~\ref{l:contraction} and Lemma~\ref{l:expansion} suggest how to have
a \emph{well-behaved} noise, namely a random perturbation that leads
the system from a ball in $H$ to a pre--blow--up set \eqref{e:bu_set}.
Let $c_0>0$, $t_0>0$, $T_c>0$, $T_e>0$ and $\psi:[0,T_e]\to\R$
be a non--negative non--increasing function, and define the event
\[
  \mathcal{N}(t_0;c_0,T_c,T_e,\psi)
    = \mathcal{N}_c(c_0,t_0,T_c)\cap \mathcal{N}_e(c_0,t_0+T_c,T_e,\psi)
\]
where
\[
  \mathcal{N}_c(c_0,t_0,T_c)
    = \bigl\{
        \lambda_{n-1}^{\beta-2}|Z_n^c(t)|\leq c_0\nu\text{ for all }n\geq1\text{ and }t\in[t_0,t_0+T_c]      
      \bigr\}\\
\]
and
\begin{multline*}
  \mathcal{N}_e(c_0,t_0+T_c,T_e,\psi)
      = \Bigl\{
          \sup_{[t_0+T_c,t_0+T_c+T_e]}\lambda_{m-1}^{\beta-2}|Z_m^e(t) - \psi_{t_0+T_c}(t)|\leq c_0\nu,\\
          \lambda_{n-1}^{\beta-2}|Z_n^e(t)|\leq c_0\nu\text{ for all }n\neq m\text{ and }t\in[t_0+T_c,t_0+T_c+T_e]
        \Bigr\},
\end{multline*}
$\psi_s:[s,s+T_e]\to\R$ is defined for $s\geq0$ as $\psi_s(t) = \psi(t-s)$,
for $t\in[s,s+T_e]$, $m$ is the smallest integer of the set
$\{n:\sigma_n\neq0\}$, and for every $n\geq1$,
\[
  \begin{gathered}
    Z^c_n(t)
      = \sigma_n\int_{t_0}^t \e^{-\nu\lambda_n^2(t-t_0)}\,dW_n,
      \qquad t\in[t_0,t_0+T_c],\\
    Z^e_n(t)
      = \sigma_n\int_{t_0+T_c}^t \e^{-\nu\lambda_n^2(t-t_0-T_c)}\,dW_n,
      \qquad t\in[t_0+T_c,t_0+T_c+T_e].
  \end{gathered}
\]
Given a weak martingale solution $\Prob_x$ with initial condition $x$, the
following facts are straightforward:
\begin{itemize}
  \item $\mathcal{N}_c(c_0,t_0,T_c)$ and $\mathcal{N}_e(c_0,t_0+T_c,T_e,\psi)$
    are independent,
  \item they both have positive probability (Lemma~\ref{l:moments_N} ensures
    that $\mathcal{N}_c$ has positive probability and the same methods can be
    used for $\mathcal{N}_e$),
  \item the value of their probability is independent of $t_0$,
  \item if $t_0,T_c,T_e$ and $t_0'$ are given such that $t_0+T_c+T_e\leq t_0'$,
    then $\mathcal{N}(t_0;c_0,T_c,T_e,\psi)$ and
    $\mathcal{N}(t_0';c_0,T_c,T_e,\psi)$ are independent.
\end{itemize}

The basic idea for the proof of Theorem~\ref{t:noway} is to show
that a pre--blow--up set \eqref{e:bu_set} is recurrent for the
system, knowing that there is no blow--up. Since a \emph{well-behaved}
randomness leads the system from a ball in $H$ to a pre--blow--up set,
it is sufficient to show that the fact that the system hits a ball
in $H$, of some radius, and \emph{afterwards} the randomness is
\emph{well-behaved} happens with probability one.
\begin{lemma}\label{l:recurrent}
  Let $\beta>3$, assume \eqref{e:noise} and \eqref{e:noise2}
  and consider $\alpha\in(\beta-2,\alpha_0+1)$.
  There exists $\cref{l:recurrent}>0$ such that
  if $M>0$, $T_c>0$, $T_e>0$, $c_0>0$, and $\psi:[0,T_e]\to\R$ is a
  non--decreasing and non--negative function and
  \[
    \frac{\cref{l:recurrent}}{M^2} + \e^{-\nu\lambda^2(T_c+T_e)}
      < 1,
  \]
  then for every $x\in V_\alpha$ and every energy martingale solution $\Prob_x$
  starting at $x$,
  \[
    \Prob_x\Bigl[\{\tau_\infty^\alpha=\infty\}
        \cap\bigcap_{k\geq1}\Bigr(
        \{\|X(kT)\|_H\leq M\}\cap\mathcal{N}(k T;c_0,T_c,T_e,\psi)\Bigl)^c\Bigr]
      = 0,
  \]
  where $T = T_c + T_e$.
\end{lemma}
\begin{proof}
  We first obtain a quantitative estimate on the return time in balls of $H$
  of the Markov process $X^R(\cdot;x)$, solution of problem \eqref{e:cutoff}
  with initial condition $x\in V_\alpha$, which has been introduced in the
  proof of Theorem~\ref{t:strong}, and then get the same estimate for the strong
  solution. This will allow us to prove the lemma.
  
  \noindent\emph{Step 1.}
  It is straightforward to prove by It\^o's formula (applied
  to $\|X^R(t;x)\|_H^2$) and Gronwall's lemma that
  \begin{equation}\label{e:expo}
    \E[\|X^R(t;x)\|_H^2]
      \leq \|x\|_H^2\e^{-2\nu\lambda^2t} + \cref{l:recurrent},
  \end{equation}
  where $\cref{l:recurrent}=(2\nu\lambda^2)^{-1}\sum_{n=1}^\infty\sigma_n^2$,
  and where the last series is convergent due to \eqref{e:noise} and to the assumption
  $\alpha_0>\beta-3$.
  
  \noindent\emph{Step 2.}
  We use the previous estimate to show that
  \begin{equation}\label{e:decay}
    \Prob\bigl[\|X^R(kT;x)\|_H\geq M\text{ for }k=1,\dots,n\bigr]
      \leq \Bigl(\e^{-\nu\lambda^2T} + \frac{\cref{l:recurrent}}{M^2}\Bigr)^{n-1}.
  \end{equation}
  We proceed for instance as in \cite[Lemma III.2.4]{Deb11}, we detail
  the proof for the sake of completeness. Define, for $k$ integer, 
  $C_k = \{\|X^R(kT;x)\|_H\geq M\}$ and $B_k = \bigcap_{j=0}^k C_j$,
  and set $\alpha_k = \E[\uno_{B_k}\|X^R(kT;x)\|_H^2]$
  and $p_k = \Prob[B_k]$.
  By the Markov property, Chebychev's inequality and \eqref{e:expo},
  \[
    \Prob[C_{k+1}|\field_{kT}]
      \leq \frac1{M^2}\e^{-\nu\lambda^2T}\|X^R(kT;x)\|_H^2
           + \frac{\cref{l:recurrent}}{M^2},
  \]
  hence
  \[
    p_{k+1}
      =    \E\bigl[\uno_{B_k}\Prob[C_{k+1}|\field_{kT}]\bigr]
      \leq \frac1{M^2}\e^{-\nu\lambda^2T}\alpha_k
           + \frac{\cref{l:recurrent}}{M^2}p_k.
  \]
  On the other hand, by integrating \eqref{e:expo} on $B_k$, we get
  \[
    \alpha_{k+1}
      \leq \E[\uno_{B_k}\|X^R((k+1)T;x)\|_H^2]
      \leq \e^{-\nu\lambda^2T}\alpha_k + \cref{l:recurrent}p_k
  \]
  If $(\widetilde\alpha_k)_{k\in\N}$ and $(\widetilde p_k)_{k\in\N}$
  are the solutions to the recurrence system
  \[
    \begin{cases}
      \widetilde\alpha_{k+1}
        = \e^{-\nu\lambda^2T}\widetilde\alpha_k + \cref{l:recurrent}\widetilde p_k,\\
      \widetilde p_{k+1}
        = \frac1{M^2}\e^{-\nu\lambda^2T}\widetilde\alpha_k + \frac{\cref{l:recurrent}}{M^2}\widetilde p_k,
    \end{cases}
    \qquad k\geq1,
  \]
  with $\widetilde p_1 = p_1$ and $\widetilde\alpha_1 = \alpha_1$,
  then $\alpha_k\leq\widetilde\alpha_k$ and $p_k\leq\widetilde p_k$
  for all $k\geq1$ and \eqref{e:decay} easily follows, since
  $\widetilde{\alpha}_k = M^2\widetilde p_k$ for $k\geq2$.
  
  \noindent\emph{Step 3.}
  By the previous step,
  \[
    \begin{aligned}
      \lefteqn{\Prob[\|X(kT;x)\|_H\geq M, k=1,\dots,n,\tau_x^\alpha=\infty]}\qquad\\
        &\leq\lim_{R\uparrow\infty}\Prob[\|X(kT;x)\|_H\geq M, k=1,\dots,n,\tau_x^{\alpha,R}>nT]\\
        &=   \lim_{R\uparrow\infty}\Prob[\|X^R(kT;x)\|_H\geq M, k=1,\dots,n,\tau_x^{\alpha,R}>nT]\\
        &\leq\limsup_{R\uparrow\infty}\Prob[\|X^R(kT;x)\|_H\geq M, k=1,\dots,n]\\
        &\leq\Bigl(\e^{-\nu\lambda^2T} + \frac{\cref{l:recurrent}}{M^2}\Bigr)^{n-1},
    \end{aligned}
  \]  
  since $\tau_x^\alpha = \sup_{R>0}\tau_x^{\alpha,R}$. Hence, if we define
  the (discrete) hitting time $K_1 = \min\{k\geq0: \|X(kT;x)||_H\leq M\}$
  of the ball $B_M(0)$ in $H$ (and $K_1=\infty$ if the set is empty), then
  $K_1<\infty$ on the event $\{\tau_x^\alpha=\infty\}$. Likewise,
  define the sequence of (discrete) return times $K_2 = \min\{k>K_1: \|X(kT;x)||_H\leq M\}$,
  \dots, $K_j = \min\{k>K_{j-1}: \|X(kT;x)||_H\leq M\}$, \dots,
  with the understanding that a return time is infinite if the corresponding
  set is empty (this in particular happens if some previous return time is
  already infinite). The previous step immediately implies that $K_j<\infty$
  for each $j\geq1$ on the event $\{\tau_x^\alpha=\infty\}$.
  
  \noindent\emph{Step 4.}
  Consider for $k\geq 1$ the events
  $\mathcal{N}_k = \mathcal{N}(kT;c_0,T_c,T_e,\psi)$. By the considerations
  we have stated above, we know that $\Prob[\mathcal{N}_k]$ is constant in $k$,
  and we set $p=\Prob[\mathcal{N}_k]$. Moreover, by the choice of $T$,
  it turns out that $\mathcal{N}_1, \mathcal{N}_2, \dots, \mathcal{N}_k, \dots$
  are all independent. Set $\mathcal{N}_\infty = \emptyset$ and define the
  time
  \[
    L_0
      = \min\{j\geq1: \uno_{\mathcal{N}_{K_j}} = 1\},
  \]
  and $L_0 = \infty$ if the set is empty. Notice that if $L_0$ is finite,
  then $\|X(K_{L_0}T;x)\|_H\leq M$ and the random perturbation
  is \emph{well-behaved} in the time interval $[K_{L_0}T, K_{L_0}T + T_c + T_e]$.
  Hence the lemma is proved if we show that
  \begin{equation}\label{e:claim}
    \Prob[L_0 = \infty, \tau_x^\alpha=\infty]
      = 0.
  \end{equation}
  
  \noindent\emph{Step 5.}
  Given an integer $\ell\geq1$, we have that
  \[
    \begin{aligned}
    \Prob[L_0>\ell, \tau_x^\alpha=\infty]
      &= \Prob[\mathcal{N}_{K_1}^c\cap\dots\mathcal{N}_{K_\ell}^c\cap\{\tau_x^\alpha=\infty\}]\\
      &= \sum_{k_1=1}^\infty\dots\sum_{k_\ell=k_{\ell-1}+1}^\infty \Prob[S_\ell(k_1,\dots,k_l)\cap\{\tau_x^\alpha=\infty\}],
    \end{aligned}
  \]
  where we have set
  \[
    S_\ell(k_1,\dots,k_\ell)
      = \mathcal{N}_{K_1}^c\cap\dots\mathcal{N}_{K_\ell}^c
        \cap\{K_1=k_1,\dots,K_\ell=k_\ell\}.
  \]
  Notice that $S_{\ell}(k_1,\dots,k_\ell)\in\field_{(k_\ell+1)T}$,
  hence by the Markov property,
  \[
    \begin{aligned}
      \lefteqn{\Prob[S_\ell(k_1,\dots,k_l)\cap\{\tau_x^\alpha>(k_\ell+1)T\}]}\qquad\\
        &= \E\bigl[\uno_{S_{\ell-1}(k_1,\dots,k_{\ell-1})}
             \uno_{\{\tau_x^\alpha>(k_{\ell-1}+1)T\}} \uno_{\{K_\ell=k_\ell\}}
             \Prob[\mathcal{N}_{k_\ell}^c\cap\{\tau_{X(k_\ell T;x)}^\alpha>T\}|\field_{k_\ell T}]\bigr]\\
        &\leq(1-p)\Prob[S_{\ell-1}(k_1,\dots,k_{\ell-1})\cap
               \{\tau_x^\alpha>(k_{\ell-1}+1)T\}\cap\{K_\ell=k_\ell\}].
    \end{aligned}
  \]
  By summing up over $k_\ell$, we have
  \begin{multline*}
    \sum_{k_\ell=k_{\ell-1}+1}^\infty \Prob[S_\ell(k_1,\dots,k_\ell)\cap\{\tau_x^\alpha>(k_\ell+1)T\}]\leq\\
      \leq (1-p) \Prob[S_{\ell-1}(k_1,\dots,k_{\ell-1})\cap\{\tau_x^\alpha>(k_{\ell-1}+1)T\}]
  \end{multline*}
  and the iteration of this argument finally yields
  \[
    \Prob[L_0>\ell, \tau_x^\alpha=\infty]
      \leq (1-p)^\ell,
  \]
  hence the claim \eqref{e:claim}.
\end{proof}
We have all the ingredients to complete the proof of the main theorem of this section.
\begin{proof}[Proof of Theorem~\ref{t:noway}]
  Fix $\alpha\in(\beta-2,1+\alpha_0)$, $\overline p\in(0,\beta-3)$ and
  $\overline{a_0}\in(0,\tfrac14]$, and consider the values $\overline{p_0}>0$,
  and $\overline{M_0}>0$ given by Theorem~\ref{t:blowup}.
  In view of Corollary~\ref{c:recurrent}, it is enough to prove that the (sampled)
  arrival time to $B_\infty(\alpha,\overline p, \overline{a_0}, \overline{M_0})$
  is finite on the event $\{\tau_x^\alpha=\infty\}$, for all $x\in V_\alpha$.
  To prove this, by virtue of Lemma~\ref{l:recurrent}, it is sufficient to prove
  that there are $M,T_c,T_e,c_0>0$ and a function $\psi$ as in the statement
  of the lemma such that $\e^{-\nu\lambda^2T_c}
  +\frac{\cref{l:recurrent}}{M^2}<1$ and
  \begin{equation}\label{e:claim2}
    \left.
      \begin{array}{c}
        \|X(t_0;x)\|_H\leq M\\
        \mathcal{N}(t_0;c_0,T_c,T_e,\psi)
      \end{array}
    \right\}
    \quad\Rightarrow\quad
    X(t_0+T_c+T_e;x)\in B_\infty(\alpha,\overline p, \overline{a_0}, \overline{M_0}).
  \end{equation}
  Indeed, the left--hand side of the above implication happens almost surely
  on $\{\tau_x^\alpha<\infty\}$ for some integer $k$ such that $t_0 = k(T_c+T_e)$,
  hence the right--hand side happens with the same probability and
  $\Prob[\sigma_{B_\infty}^{x,T_c+T_e}=\infty,\tau_x^\infty=\infty] = 0$.

  We conclude with the proof of \eqref{e:claim2}. We first notice that in
  Lemma~\ref{l:contraction}, the larger we choose $M$, the larger is the time $T_c$
  we get, hence we apply Lemma~\ref{l:contraction} with $a_0 = \frac18\overline{a_0}$,
  $c_0<\min\{\frac18\overline{a_0}, (4(1+\lambda^{\beta-3}))^{-1}\}$ and $M>0$
  large enough so that the time $T_c$ whose existence is ensured by the lemma
  satisfies $\e^{-\nu\lambda^2T_c} +\frac{\cref{l:recurrent}}{M^2}<1$.
  Moreover we know that
  \begin{itemize}
    \item $\inf_{n\geq 1}\bigl(\lambda_{n-1}^{\beta-2}X_n(t_0+T_c)\bigr)
     \geq -(a_0 + c_0)\nu\geq -\frac14 \overline{a_0}\nu$ on the event
      $\{\tau_x^\alpha=\infty\}\cap\mathcal{N}_c(c_0,t_0,T_c)$,
    \item $\|X(t_0+T_c)\|_H\leq \|X(t_0)\|_H + \cref{l:Hcontrol}\nu\leq M + \cref{l:Hcontrol}\nu$,
  \end{itemize}
  where the second statement follows from Lemma~\ref{l:Hcontrol}. We apply
  then Lemma~\ref{l:expansion} with $M_1=M + \cref{l:Hcontrol}\nu$,
  $M_2=\overline{M_0}$, $a_0 = \frac14\overline{a_0}$, $a_0' = 2a_0$
  and $c_0$ as above, hence there is $T_e>0$ such that
  \begin{itemize}
    \item $\inf_{n\geq 1}\bigl(\lambda_{n-1}^{\beta-2}X_n(t_0+T_c+T_e)\bigr)
     \geq -(a_0' + c_0)\nu\geq - \overline{a_0}\nu$ on the event
      $\{\tau_x^\alpha=\infty\}\cap\mathcal{N}_e(c_0,t_0+T_c,T_e,\psi)$,
    \item $\|X(t_0+T_c+T_e)\|_{\overline p}\geq \lambda^{\overline p}\|X(t_0+T_c+T_e)\|_H\geq \overline{M_0}$,
  \end{itemize}
  that is $X(t_0+T_c+T_e)\in B_\infty(\alpha,\overline p,\overline{a_0},
  \overline{M_0})$.
\end{proof}
\bibliographystyle{amsplain}

\end{document}